\documentclass[11pt,a4paper]{article}
\usepackage{amssymb,amsmath,mathrsfs,enumerate,comment}
\allowdisplaybreaks[1]
\numberwithin{equation}{section}

\usepackage[colorlinks=true, pdfstartview=FitV, linkcolor=blue, citecolor=blue, urlcolor=blue,pagebackref=false]{hyperref}

\usepackage{tikz}
\usepackage{float}
\usepackage[font=footnotesize]{caption}

\usepackage{color} 
\usepackage{times,theorem,latexsym,color,comment}

\newcommand{\BOX}{\ensuremath\Box}

\newtheorem{theorem}{Theorem }[section]

\newtheorem{corollary}[theorem]{Corollary}

{\theorembodyfont{\rmfamily}}
{\theorembodyfont{\rmfamily}}
{\theorembodyfont{\rmfamily}}
\newtheorem{lemma}[theorem]{Lemma}
\newtheorem{proposition}[theorem]{Proposition}
{\theorembodyfont{\rmfamily}\newtheorem{remark}[theorem]{Remark}}
{\theorembodyfont{\rmfamily}}

\newcommand{\Z}{\mathbb{Z}}
\newcommand{\R}{\mathbb{R}}
\newcommand{\C}{\mathbb{C}}
\newcommand{\dd}{\,{\rm d}}
\newcommand{\opind}{\operatorname{ind}}
\newcommand{\opdim}{\operatorname{dim}}
\newcommand{\opdiv}{\operatorname{div}}
\newcommand{\oprot}{\operatorname{rot}}
\newcommand{\opsgn}{\operatorname{sgn}}
\newcommand{\oparg}{\operatorname{arg}}
\newcommand{\opLog}{\operatorname{Log}}
\newcommand{\ep}{\epsilon}

\def\XXint#1#2#3{{\setbox0=\hbox{$#1{#2#3}{\int}$}
		\vcenter{\hbox{$#2#3$}}\kern-.5\wd0}}

\DeclareMathOperator*{\esssup}{ess\,sup}

\newenvironment{proof}{{\vskip\baselineskip\noindent\textbf{Proof:}}}%
{\hspace*{.1pt}\hspace*{\fill}\BOX\vskip\baselineskip}

\newenvironment{proofx}[1]%
{\vskip\baselineskip\noindent\textbf{Proof of {#1}:}}%
{\hspace*{.1pt}\hspace*{\fill}\BOX\vskip\baselineskip}
{\vskip\baselineskip\noindent\textbf{Proof of Theorem \protect\ref{#1}:}}%
{\hspace*{.1pt}\hspace*{\fill}\BOX\vskip\baselineskip}
{\vskip\baselineskip\noindent\textbf{Proof of Theorems \protect\ref{#1} --
		\protect\ref{#2}:}}%
{\hspace*{.1pt}\hspace*{\fill}\BOX\vskip\baselineskip}

\begin{document}

\title{Stability of planar exterior stationary flows with suction}

\author{Mitsuo Higaki \\
Department of Mathematics, 
Graduate School of Science, 
Kobe University, \\
1-1 Rokkodai, Nada-ku, Kobe 657-8501, Japan \\
E-mail: higaki@math.kobe-u.ac.jp}
\date{}

\maketitle

\noindent {\bf Abstract.}\ 
We consider the two-dimensional Navier-Stokes system in a domain exterior to a disk. The system admits a stationary solution with critical decay $O(|x|^{-1})$ written as a linear combination of the pure rotating flow and the flux carrier. We prove its nonlinear stability in large time for initial disturbances in $L^2$ under smallness conditions, assuming that there is suction across the boundary, namely that the sign of coefficients of the flux carrier is negative. This result partially solves an open problem in the literature.

\medskip


\medskip


\tableofcontents

\section{Introduction}\label{sec.intro.}

We consider the two-dimensional Navier-Stokes system in an exterior disk  
\begin{equation}\tag{NS}\label{eq.NS}
\left\{
\begin{array}{ll}
\partial_t u - \Delta u + \nabla p 
= -u\cdot\nabla u&\mbox{in}\ (0,\infty)\times\Omega \\
\opdiv u =0&\mbox{in}\ [0,\infty)\times\Omega \\
u(x)=\alpha x^{\bot} - \delta x&\mbox{on}\ (0,\infty)\times\partial\Omega \\
u|_{t=0}=u_0&\mbox{in}\ \Omega. 
\end{array}\right.
\end{equation}
The unknown functions $u=(u_1(t,x),u_2(t,x))$ and $p=p(t,x)$ are respectively the velocity field of the fluid and the pressure field. The function $u_0=(u_{0,1}(x), u_{0,2}(x))$ is a given initial data. The set $\Omega$ denotes the exterior unit disk $\{x=(x_1,x_2)\in\R^2~|~|x|>1\}$ where $|x| = \sqrt{x_1^2+x_2^2}$. We assume that both $\alpha$ and $\delta$ are real number constants. The vector $x^\bot$ refers to $(-x_2,x_1)$. The system \eqref{eq.NS} describes the time evolution of viscous incompressible fluids around the disk rotating at angular velocity $\alpha$ on whose surface there is suction in the orthogonal direction when $\delta>0$ and injection when $\delta<0$.

The system \eqref{eq.NS} admits an explicit stationary solution $(\alpha U -\delta W, \nabla P_{\alpha,\delta})$ where   
\begin{align}\label{def.U.W}
\begin{split}
U(x) 
= \frac{x^\bot}{|x|^2}, 
\qquad
W(x) 
= \frac{x}{|x|^2}, 
\end{split}
\end{align}
and 
\begin{align}\label{def.P}
\begin{split}
\nabla P_{\alpha,\delta}(x) = -\nabla\Big(\frac{|\alpha U(x) - \delta W(x)|^2}{2}\Big).
\end{split}
\end{align}
This velocity is a linear combination of the vector field $U$ denoting the pure rotating flow in $\Omega$ and $W$ the flux carrier. To lighten notation, in the following, we write  
\begin{align}\label{def.V}
V = V(\alpha,\delta) = \alpha U - \delta W. 
\end{align}
The solution $V$ is invariant under the scaling of the Navier-Stokes equations. A (non-trivial) solution having this property is said to be scale-critical and it represents the balance between the nonlinear and linear parts of the equations. Therefore, investigating the properties of the scale-critical solutions is a fundamental and important issue in understanding the typical behavior of the Navier-Stokes flows. Let us mention that $V$ is an element of the family of stationary solutions of \eqref{eq.NS} found by Hamel \cite{Hamel(1917)}. This family is known to be an example showing the non-uniqueness of the $D$-solutions; see Galdi \cite[Section X\hspace{-.1em}I\hspace{-.1em}I.2]{Galdi(2011)}. The Hamel solutions are generalized by Guillod and Wittwer \cite{Guillod-Wittwer(2015)} in view of rotation symmetries.

In this paper, we study the nonlinear stability of $V$ in large time. More precisely, assuming that an initial disturbance around $V$ belongs to the Lebesgue spaces, we consider the time evolution of the disturbance in the nonlinear system \eqref{eq.NS}. Particularly, we are interested in the large-time decay estimate. By using the relation 
\begin{align}\label{eq.bilinear.rot}
u\cdot\nabla v + v\cdot\nabla u 
= u^{\bot} {\rm rot}\,v + v^{\bot} {\rm rot}\,u 
+ \nabla\Big(\frac{|u+v|^2-|u|^2-|v|^2}2\Big) 
\end{align}
and $\oprot V=0$, we see that the pair of new unknown functions 
$$
v = u - V 
\qquad
\text{and}
\qquad 
\nabla q = \nabla\Big(p + \frac{|u|^2}2\Big) 
$$
formally solves the nonlinear problem 
\begin{equation}\tag{NP}\label{eq.nonlin.prob}
\left\{
\begin{array}{ll}
\partial_t v - \Delta v + V^\bot \oprot v + \nabla q 
= -v^{\bot} \oprot v&\mbox{in}\ (0,\infty)\times\Omega \\
\opdiv v=0&\mbox{in}\ [0,\infty)\times\Omega \\
v=0&\mbox{on}\ (0,\infty)\times\partial\Omega \\
v|_{t=0}=v_0 := u|_{t=0} - V&\mbox{in}\ \Omega. 
\end{array}\right.
\end{equation}
The linearized problem of \eqref{eq.nonlin.prob} is given by 
\begin{equation}\tag{LP}\label{eq.lin.prob}
\left\{
\begin{array}{ll}
\partial_t v - \Delta v + V^\bot \oprot v + \nabla q 
=0&\mbox{in}\ (0,\infty)\times\Omega \\
\opdiv v=0&\mbox{in}\ [0,\infty)\times\Omega \\
v=0&\mbox{on}\ (0,\infty)\times\partial\Omega \\
v|_{t=0}=v_0&\mbox{in}\ \Omega. 
\end{array}\right.
\end{equation}
Our main aim in this paper is to obtain large-time decay estimates of the solutions of \eqref{eq.lin.prob}, by studying the operators associated with \eqref{eq.lin.prob}. We will provide the $L^p$-$L^q$ estimates sufficient to prove the nonlinear stability of the stationary solution $V$ in large time.

In order to make the framework clearer, we recall some notations and basic facts about the linear system \eqref{eq.lin.prob}. We let $C^\infty_{0,\sigma}(\Omega)$ denote $\{\varphi\in C^\infty_0(\Omega)^2~|~\opdiv\varphi=0\}$, $L^2_\sigma(\Omega)$ the closure of $C^\infty_{0,\sigma}(\Omega)$ in $L^2(\Omega)^2$, and ${\mathbb P}: L^2(\Omega)^2 \to L^2_\sigma(\Omega)$ the orthogonal projection. The operator ${\mathbb P}$ is called the Helmholtz projection and satisfies ${\mathbb P} \nabla p=0$ for $p\in L^2_{{\rm loc}}(\overline{\Omega})$ with $\nabla p\in L^2(\Omega)^2$. The operator, called the Stokes operator, is defined by 
$$
{\mathbb A} = -{\mathbb P} \Delta,
\qquad
D({\mathbb A}) = L^{2}_{\sigma}(\Omega) \cap W^{1,2}_0(\Omega)^2 \cap W^{2,2}(\Omega)^2. 
$$
It is well known that ${\mathbb A}$ is nonnegative and self-adjoint in $L^{2}_{\sigma}(\Omega)$ and that $-{\mathbb A}$ generates the $C_0$-analytic semigroup; see Sohr \cite{Sohr(2013)}. Moreover, the spectrum of $-{\mathbb A}$ is the set of nonpositive real numbers $\sigma(-{\mathbb A})=\R_{\le0}=\{x\in\R~|~x\le0\}$; see Section \ref {sec.spec.anal.} for the references. With these notations, we define the operator associated with \eqref{eq.lin.prob} by 
$$
{\mathbb A}_V v 
= {\mathbb A} v 
+ {\mathbb P} V^\bot \oprot v, 
\qquad 
D({\mathbb A}_V) = D({\mathbb A}), 
$$
and write \eqref{eq.lin.prob} equivalently with the evolution system 
\begin{equation}\label{intro.evel.sys}
\frac{\dd v}{\dd t} + {\mathbb A}_V v 
= 0 
\mkern9mu 
\mbox{in}\ (0,\infty), 
\qquad 
v|_{t=0} = v_0. 
\end{equation}
We aim at proving the properties of solutions of \eqref{intro.evel.sys} by studying the operator $-{\mathbb A}_V$.

One basic way to study the properties of $-{\mathbb A}_V$ is to consider the equation 
\begin{equation}\tag{R}\label{intro.eq.resol}
(\lambda + {\mathbb A}_V)v = f
\end{equation}
for given $\lambda\in\C$ and $f\in L^2_{\sigma}(\Omega)$. This equation can be obtained by formal application of the Laplace transform to \eqref{intro.evel.sys}. From the general theory of functional analysis, we find the following two facts. First, as the operator ${\mathbb P} V^\bot \oprot$ is lower order with respect to ${\mathbb A}$, from theory for sectorial operators, we see that $-{\mathbb A}_V$ is also sectorial in $L^2_{\sigma}(\Omega)$ and generates the $C_0$-analytic semigroup, denoted by $\{e^{-t {\mathbb A}_V}\}_{t\ge0}$; see Lunardi \cite[Proposition 2.4.3]{Lunardi(1995)}. Second, as ${\mathbb P} V^\bot \oprot$ is relatively compact with respect to ${\mathbb A}$, from the perturbation theory of operators, we see that $\sigma(-{\mathbb A}_V)=\R_{\le0}\cup\sigma_{{\rm disc}}(-{\mathbb A}_V)$ where $\sigma_{{\rm disc}}(-{\mathbb A}_V)$ denotes the discrete spectrum of $-{\mathbb A}_V$; see Section \ref {sec.spec.anal.} for details. These two facts, however, are not sufficient to obtain the large-time estimate of $\{e^{-t {\mathbb A}_V}\}_{t\ge0}$ since $\sigma(-{\mathbb A}_V)$ contains $\R_{\le 0}$. We need a precise estimate of the resolvent $(\lambda + {\mathbb A}_V)^{-1}$ when $\lambda$ is close to the origin.

The fundamental difficulty in analyzing \eqref{intro.eq.resol} when $|\lambda|\ll1$ is that the Hardy inequality 
\begin{align}\label{ineq.hardy}
\Big\| 
x
\mapsto
\frac{f(x)}{|x|} 
\Big\|_{L^2} 
\le C \|\nabla f\|_{L^2}, 
\quad 
f \in 
\dot{W}^{1,2}_0 (\Omega)^d
= \Big(\overline{C^\infty_0 (\Omega)}^{\|\nabla\, \cdot\,\|_{L^2}}\Big)^d 
\end{align}
does not hold in exterior domains $\Omega\subset \R^d$ when $d=2$. If \eqref{ineq.hardy} holds when $d=2$, the term ${\mathbb P} V^\bot \oprot v$ in \eqref{intro.eq.resol} can be controlled by the dissipation from $-\Delta v$ if $|\alpha|+|\delta|$ is small. Nevertheless, one needs a logarithmic correction in the left-hand side of \eqref{ineq.hardy} to obtain the correspondence; see \cite[Theorem II.6.1]{Galdi(2011)}. This implies that energy method does not work well in general in deriving estimates for \eqref{intro.eq.resol} when $|\lambda|\ll1$. One way to recover inequalities of the type \eqref{ineq.hardy} when $d=2$ is to assume symmetries both on $\Omega$ and $f$; see Galdi and Yamazaki \cite{Galdi-Yamazaki(2015)}, Yamazaki \cite{Yamazaki(2016)}, and Guillod \cite{Guillod(2017)} for the stability results of symmetric flows under symmetries. As we do not assume any symmetries on initial data in \eqref{eq.nonlin.prob}, unlike \cite{Galdi-Yamazaki(2015), Yamazaki(2016), Guillod(2017)}, such inequalities are not applicable to \eqref{eq.lin.prob} nor \eqref{intro.eq.resol}. This is in stark contrast to the three-dimensional stability results by Heywood \cite{Heywood(1970)} and by Borchers and Miyakawa \cite{Borchers-Miyakawa(1992), Borchers-Miyakawa(1995)} in which the Hardy inequality \eqref{ineq.hardy} with $d=3$ is an essential tool. As a recent monograph of the three-dimensional results, we refer to Brandolese and Schonbek \cite{Brandolese-Schonbek(2018)}.

Therefore, even for the flow $V=V(\alpha,\delta)$ explicitly given in \eqref{def.V}, the stability analysis in two-dimensional exterior domains requires specific considerations depending on the parameters $\alpha$ and $\delta$. The known results are summarized as follows.

\begin{itemize}
\item The case $\alpha=0$ and $\delta\neq0$ is treated in Guillod \cite{Guillod(2017)}. This case is tractable and similar to the three-dimensional cases if $|\delta|$ is sufficiently small. In fact, for general exterior domains $\Omega\subset \R^2$, Russo \cite[Lemma 3]{Russo(2011)} proves the Hardy-type inequality 
\begin{align}\label{Hardytype.ineq.W}
|\langle u\cdot\nabla u, W \rangle|
\le 
C \|\nabla u\|_{L^2}^2, 
\quad 
u \in 
\dot{W}^{1,2}_{0,\sigma} (\Omega)
= \overline{C^\infty_{0,\sigma} (\Omega)}^{\|\nabla\, \cdot\,\|_{L^2}}. 
\end{align}
The reader is referred to \cite[Remark X.4.2]{Galdi(2011)} and \cite[Lemma 3]{Guillod(2017)} for further discussions. Combining \eqref{Hardytype.ineq.W} with the relation \eqref{eq.bilinear.rot}, we obtain the control 
\begin{align}\label{Hardytype.ineq.W.2}
|\langle {\mathbb P} (\delta W)^\bot \oprot v, v\rangle|
\le 
C |\delta| 
\|\nabla v\|_{L^2}^2, 
\quad 
v\in D({\mathbb A}_V). 
\end{align}
This observation implies that, by a simple energy estimate applied to \eqref{intro.eq.resol}, we can obtain the $L^p$-$L^q$ estimates for the system \eqref{eq.lin.prob} and prove the nonlinear stability of $V=\delta W$. Alternatively, as is done in \cite{Guillod(2017)}, one can prove the stability by considering $L^2$-estimates of the semigroup generated by the adjoint of the operator $-{\mathbb A}_{\delta W}$. A similar idea is also used in Karch and Pilarczyk \cite{Karch-Pilarczyk(2011)}.

\item The case $\alpha\neq0$ and $\delta=0$ is treated in Maekawa \cite{Maekawa(2017a)}. In this case, energy method is not useful for \eqref{intro.eq.resol}. Indeed, \cite[Lemma 4]{Guillod(2017)} points out that the Hardy-type inequality \eqref{Hardytype.ineq.W} does not hold if $W$ is replaced by $U$. To relax the situation, \cite{Maekawa(2017a)} considers the problem in an exterior disk and performs explicit computations. The $L^p$-$L^q$ estimates for \eqref{eq.lin.prob} are obtained when $|\alpha|$ is sufficiently small by an explicit formula for the resolvent $(\lambda + {\mathbb A}_{\alpha U})^{-1}$. Also, the nonlinear stability of $\alpha U$ is proved when both $|\alpha|$ and the $L^2$-norm of initial data in \eqref{eq.nonlin.prob} are sufficiently small. This stability result is extended by the author in \cite{Higaki(2019)} to a certain class of non-symmetric domains where the domains are assumed to be small perturbations of the exterior unit disk, and in \cite{Higaki(2023)} for three-dimensional initial disturbances around an infinite cylinder.

\item The case $\alpha\neq0$ and $\delta\neq0$ is treated in Maekawa \cite{Maekawa(2017b)}. The problem is considered on an exterior disk as in \cite{Maekawa(2017a)}. The idea of the proof is to regard the term ${\mathbb P} (\delta W)^\bot \oprot v$ in \eqref{intro.eq.resol} as an external force and to utilize the estimate of $(\lambda + {\mathbb A}_{\alpha U})^{-1}$ in \cite{Maekawa(2017a)}. The $L^p$-$L^q$ estimates for \eqref{eq.lin.prob} are obtained when $|\alpha|+|\delta|$ is sufficiently small, under the restriction that initial data belong to a subcritical space $L^2\cap L^q$ for some $1<q<2$. Also, the nonlinear stability of $V$ is proved when both $|\alpha|+|\delta|$ and the ($L^2\cap L^q$)-norm of initial data in \eqref{eq.nonlin.prob} are sufficiently small. This restriction on exponents is essentially needed when estimating $(\lambda + {\mathbb A}_{\alpha U})^{-1} {\mathbb P} (\delta W)^\bot \oprot v$. As is mentioned in \cite[Remark 2]{Maekawa(2017b)}, it is not clear if the condition $q<2$ can be removed in this method. 
\end{itemize}
%

\subsection{Main results}

This paper addresses large-time estimates for the system \eqref{eq.lin.prob}, namely, of the semigroup $\{e^{-t {\mathbb A}_V}\}_{t\ge0}$, when $\alpha\neq0$ and $\delta\neq0$ as in \cite{Maekawa(2017b)}. Our particular interest is the $L^p$-$L^2$ estimates left as open problems in \cite{Maekawa(2017b)}. The following theorem solves it affirmatively under a condition on the sign of $\delta$. This condition is discussed in Remark \ref{rem.thm.L2L2} (\ref{item3.rem.thm.L2L2}) below. 
%
\begin{theorem}\label{thm.L2L2}
Let $\alpha,\delta\in\R$ satisfy $\alpha\neq0$ and $\delta\ge0$ and let $|\alpha|+\delta$ be sufficiently small. For $f\in L^2_\sigma(\Omega)$, we have 
\begin{align}\label{est.thm.L2L2}
\begin{split}
\|e^{-t {\mathbb A}_V} f\|_{L^2} 
& \le 
C \|f\|_{L^2}, 
\quad 
t>0, \\
\|\nabla e^{-t {\mathbb A}_V} f\|_{L^2} 
& \le 
C t^{-\frac12}
\|f\|_{L^2},
\quad 
t>0. 
\end{split}
\end{align}
The constant $C$ depends on $\alpha,\delta,p$. 
\end{theorem}
%
%
\begin{remark}\label{rem.thm.L2L2}
\begin{enumerate}[(i)]
\item\label{item1.rem.thm.L2L2}
By combining Theorem \ref{thm.L2L2} with the $L^p$-$L^q$ estimates in \cite{Maekawa(2017b)} and by applying the Gagliardo-Nirenberg inequality, we obtain 
\begin{align}\label{LpLq.rem.thm.L2L2}
\begin{split}
\|e^{-t {\mathbb A}_V} f\|_{L^p} 
& \le 
C t^{-\frac1q+\frac1p}
\|f\|_{L^q}, 
\quad 
t>0, \\
\|\nabla e^{-t {\mathbb A}_V} f\|_{L^2} 
& \le 
C t^{-\frac1q}
\|f\|_{L^q},
\quad 
t>0 
\end{split}
\end{align}
for $1<q\le2\le p<\infty$ and $f\in L^2_\sigma(\Omega) \cap L^q(\Omega)^2$ with a constant $C=C(\alpha,\delta,q,p)$.

\item\label{item2.rem.thm.L2L2}
The proof of Theorem \ref{thm.L2L2} is based on an analysis of the operator $-{\mathbb A}_V$. The estimate \eqref{est.thm.L2L2} for $\{e^{-t {\mathbb A}_V}\}_{t\ge0}$ is deduced by the Dunford integral of the resolvent $(\lambda + {\mathbb A}_V)^{-1}$. Inspired by \cite{Maekawa(2017a)}, we determine the spectrum of $-{\mathbb A}_V$ and estimate $(\lambda + {\mathbb A}_V)^{-1}$ by explicit computations. It is shown in Section \ref{sec.spec.anal.} that the function characterizing the discrete spectrum of $-{\mathbb A}_V$ crucially depends on both $\alpha$ and $\delta$. Therefore, it is suggested that, when estimating $(\lambda + {\mathbb A}_V)^{-1}$ for $|\lambda|\ll1$, one cannot regard ${\mathbb P} (\delta W)^\bot \oprot v$ in \eqref{intro.eq.resol} as an external force even if $|\delta|$ is small, in spite of the control \eqref{Hardytype.ineq.W.2}.

\item\label{item3.rem.thm.L2L2}
It is an open problem whether the same estimate as in \eqref{est.thm.L2L2} can be obtained for the case $\delta<0$. Actually, by following the argument in Section \ref{sec.quant.anal.disc.spec}, one can prove \eqref{est.thm.L2L2} if $\delta$ is chosen to depend on a given $\alpha$, but the general case is still open. It might be meaningful to recall here that the case $\delta<0$ corresponds to the situation where there is injection into fluids at the boundary. We mention Drazin and Reid \cite[Problem 3.7]{Drazin-Reid(2004)} and Drazin and Riley \cite[Section 3.1]{Drazin-Riley(2006)} as the references related to this topic.

\item\label{item4.rem.thm.L2L2}
It is important to extend the $L^p$-$L^q$ estimates in \eqref{LpLq.rem.thm.L2L2} to general exterior domains. However, this is a difficult problem because of the dependence of constants on $\alpha,\delta$. The problem when $\delta=0$ is tackled in \cite{Higaki(2019)} and it is shown that, if the domain $\Omega$ is a perturbation from the exterior unit disk in algebraic order of $|\alpha|$, then the $L^p$-$L^q$ estimates can be obtained by energy method combined with explicit formulas. The restriction to a class of domains is due to singularity in the operator norm of the resolvent $(\lambda+{\mathbb A}_{\alpha U})^{-1}$ for small $|\alpha|$. It is observed that, in explicit computations, cancellation of the effects from the two terms $\lambda v$ and $\alpha U^\bot \oprot v$ in \eqref{intro.eq.resol} (with $\delta=0$) occurs for $\lambda$ in a certain domain, dubbed the ``nearly-resonance regime" in \cite{Higaki(2019)}. This cancellation causes the singularity at algebraic order of $|\alpha|$, which in energy method restricts the shape of domains, more precisely the lengths between domains and the exterior unit disk. Such singularity also appears in the operator norm of $(\lambda+{\mathbb A}_{V})^{-1}$ for small $|\alpha|+\delta$ in the present problem and is an obstacle to the extension.
\end{enumerate}
\end{remark}
%

By using Theorem \ref{thm.L2L2}, we can prove the nonlinear stability of $V$. Using the semigroup $\{e^{-t {\mathbb A}_V}\}_{t\ge0}$, we consider the mild solutions of \eqref{eq.nonlin.prob} solving 
\begin{align}\label{eq.INS}
v(t) 
= e^{-t {\mathbb A}_V} v_0 
- \int_0^t e^{-(t-s) {\mathbb A}_V} \mathbb{P} (v^{\bot} \oprot v)(s) \dd s, 
\quad
t>0. 
\end{align}
The following theorem can be shown by a simple application of the Banach fixed point theorem and thus is omitted in this paper. For details, see \cite{Maekawa(2017a)} treating the case $\delta=0$.
%
\begin{theorem}\label{thm.nonlin.stability}
Let $\alpha,\delta\in\R$ satisfy $\alpha\neq0$ and $\delta\ge0$ and let $|\alpha|+\delta$ be sufficiently small. Let $v_0$ belong to $L^2_{\sigma}(\Omega)$ and let $\|v_0\|_{L^2}$ be sufficiently small depending on $\alpha,\delta$. There is a unique mild solution $v \in C\big([0,\infty); L^2_\sigma(\Omega)\big) \cap C\big((0,\infty);W^{1,2}_0(\Omega)^2\big)$ of \eqref{eq.INS} satisfying
\begin{align}\label{est1.thm.nonlin.stability}
\lim_{t\to\infty} t^{\frac{k}{2}} \|\nabla^k v(t)\|_{L^2} = 0,
\quad 
k = 0,1. 
\end{align}
\end{theorem}
%

\subsection{Related results}

Let us refer to the results that are closely related to the present study.

\smallskip

\noindent\underline{{\it Analysis of \eqref{eq.nonlin.prob} and \eqref{eq.lin.prob} when $V\equiv0$.}} For \eqref{eq.nonlin.prob}, the estimate \eqref{est1.thm.nonlin.stability} for $V\equiv0$, which can be viewed as the nonlinear stability of the trivial solution, is classical; see Masuda \cite{Masuda(1984)} for the proof when $k=0$ and Kozono and Ogawa \cite{Kozono-Ogawa(1993)} when $k=1$. These results do not require smallness on the initial data in $L^2_\sigma(\Omega)$. For \eqref{eq.lin.prob}, the $L^p$-$L^q$ estimates of the Stokes semigroup $\{e^{-t {\mathbb A}}\}_{t\ge0}$ are established by Maremonti and Solonnikov \cite{Maremonti-Solonnikov(1997)} and by Dan and Shibata \cite{Dan-Shibata(1999a), Dan-Shibata(1999b)}. We note that all of the results above hold in general exterior domains $\Omega\subset \R^2$. It is pointed out in \cite[Remark 1.4]{Maekawa(2017a)} that the logarithmic singularity of the resolvent $(\lambda + {\mathbb A})^{-1}$ for small $|\lambda|$, observed  in \cite[\S3]{Dan-Shibata(1999a)}, disappears in $(\lambda + {\mathbb A_{\alpha U}})^{-1}$ if $\alpha\neq0$. As compensation, however, singularity appears in the operator norm of $(\lambda + {\mathbb A_{\alpha U}})^{-1}$ for small $|\alpha|$. Such singularity, as discussed in Remark \ref{rem.thm.L2L2} (\ref{item4.rem.thm.L2L2}), also appears in the operator norm of $(\lambda+{\mathbb A}_{V})^{-1}$, and is an obstacle when generalizing the $L^p$-$L^q$ estimates in \eqref{LpLq.rem.thm.L2L2}. Let us mention the study of the boundedness of $\{e^{-t {\mathbb A}}\}_{t\ge0}$ in spaces $L^p_{\sigma}(\Omega)=\overline{C^\infty_{0,\sigma} (\Omega)}^{\|\,\cdot\,\|_{L^p}}$ pioneered by Borchers and Varnhorn \cite{Borchers-Varnhorn(1993)}. See Abe \cite{Abe(2020), Abe(2021)} for the recent progress.

\smallskip

\noindent\underline{{\it Non-symmetric stationary solutions around $V$.}}
We consider the stationary problem of \eqref{eq.NS}, which also admits the explicit solution $V$. It is known that, for suitably chosen $\alpha,\delta$, the fundamental solution for the linearized problem around $V$, namely for the stationary problem of \eqref{eq.lin.prob}, has a better spatial decay compared to the one for the problem linearized around the trivial solution $V\equiv0$. This improvement is due to the vorticity transport by $V$ and implies the resolution of the famous Stokes paradox; see \cite{Chang-Finn(1961), Galdi(2004), Galdi(2011), Kozono-Sohr(1992), Hishida(2016)} for descriptions. Furthermore, these new fundamental solutions allow us to construct non-symmetric solutions for the nonlinear problem decaying in the order $O(|x|^{-1})$. This is done in Hillairet and Wittwer \cite{Hillairet-Wittwer(2013)} when $|\alpha|>\sqrt{48}$ and $\delta=0$ for given zero-flux boundary data in a suitable class, and in \cite{Higaki(2022)} when $\alpha\in\R$ and $\delta>2$ for given external forces with suitable spatial decay. The solutions in \cite{Higaki(2022)} are compatible with the Liouville-type theorem in Guillod \cite[Proposition 4.6]{Guillod(2015)}. We emphasize that the results in \cite{Hillairet-Wittwer(2013), Higaki(2022)} do not require any symmetries on the given data. Interestingly, such improvement in the fundamental solutions occurs even for small $\alpha,\delta$. Indeed, Maekawa and Tsurumi \cite{Maekawa-Tsurumi(2023)} constructs non-symmetric solutions for the nonlinear problem in the whole space $\R^2$, whose principal part at spatial infinity is $c U$ with a small but nonzero constant $c$. This result is contrasting with \cite{Hillairet-Wittwer(2013)} in view of the size of coefficients, and the reason is that, as there are no boundaries in $\R^2$, the terms needed to match the no-slip boundary condition in exterior domains do not appear in the problem.

\subsection{Outlined proof}

We describe the proof of Theorem \ref{thm.L2L2}. However, the estimate \eqref{est.thm.L2L2} is almost a direct consequence of the estimate of the resolvent $(\lambda + {\mathbb A}_V)^{-1}$ in Proposition \ref{prop.resol.LpLq}. Hence we give in Appendix \ref{app.proof.thm.L2L2} the proof that derives Theorem \ref{thm.L2L2} from Proposition \ref{prop.resol.LpLq}, and outline here the proof of Proposition \ref{prop.resol.LpLq}. As noted in Remark \ref{rem.thm.L2L2} (\ref{item2.rem.thm.L2L2}), it consists of two steps:

\noindent {\bf (I) Spectral analysis of $-{\mathbb A}_V$.} 
Recall that $\sigma(-{\mathbb A}_V)=\R_{\le0}\cup\sigma_{{\rm disc}}(-{\mathbb A}_V)$. Thus we identify the location of the discrete spectrum $\sigma_{{\rm disc}}(-{\mathbb A}_V)$ to obtain the large-time decay of $\{e^{-t {\mathbb A}_V}\}_{t\ge0}$. For this purpose, we consider the homogeneous equation of \eqref{intro.eq.resol} and its general solutions, by using the streamfunction-vorticity equations. We see that the no-slip boundary condition imposes that $\lambda$ belongs to $\sigma_{{\rm disc}}(-{\mathbb A}_V)$ if and only if $\lambda$ belongs to 
$$
\bigcup_{n\in\Z}
\{\lambda\in\C\setminus\R_{\le0}~|~F_n(\sqrt{\lambda})=0\}. 
$$
Here $F_n=F_n(z)$ is the analytic function defined in \eqref{def.Fn} in Section \ref{sec.spec.anal.}. For $|n|\neq1$, one can show that $F_n(\sqrt{\cdot})$ has no zeros in the sector $\Sigma_{\frac{3}{4}\pi}$ by energy method if $|\alpha|+|\delta|$ is sufficiently small; see Propositions \ref{prop.spec.small} and \ref{prop.charact.disc.ev}. However, for $|n|=1$, we need to deal with the function $F_n$ directly to determine the location of its zeros, which reflects the fact that the Hardy inequality does not hold in two-dimensional exterior domains. We will prove that $F_n$ with $|n|=1$ has no zeros in sectors $\Sigma_{\frac{3}{4}\pi-\ep}$ for $\ep\in(0,\frac{\pi}4)$ if $\delta\ge0$ and $|\alpha|+\delta$ is sufficiently small depending on $\ep$. The proof is the most tricky part of this paper and will take the whole of Section \ref{sec.quant.anal.disc.spec}. We perform an asymptotic analysis of $F_n$ that refines the methods in \cite{Maekawa(2017a),Higaki(2019)}. Interestingly, the analysis is highly dependent on the sign of $\delta$ being positive or negative. Furthermore, we observe that the condition $\delta>0$, which is also an assumption of Theorem \ref{thm.L2L2}, provides a certain stabilizing effect compared to the case $\delta=0$; see Remark \ref{rem.prop.est.Fn}.

\noindent {\bf (I\hspace{-.1em}I) Estimate of the resolvent $(\lambda + {\mathbb A}_V)^{-1}$.} 
In the next step, we estimate the solution of \eqref{intro.eq.resol} for $\lambda$ belonging to the resolvent set. We derive and estimate an explicit formula for the solution using the streamfunction-vorticity equations. The computations are lengthy ones estimating the formulas involving the modified Bessel functions, but the approach itself is broadly the same as that used in \cite{Maekawa(2017a),Higaki(2019)}. Thus we omit some details; see Section \ref{sec.resol.est}.

\smallskip

This paper is organized as follows. In Section \ref{sec.preliminaries}, we collect the items used in this paper. In Sections \ref{sec.spec.anal.} and \ref{sec.quant.anal.disc.spec}, we study the spectrum of the operator $-{\mathbb A}_V$. We apply the perturbation theory of operators in Section \ref{sec.spec.anal.} and perform an asymptotic analysis of $F_n$ with $|n|=1$ in Section \ref{sec.quant.anal.disc.spec}. In Section \ref{sec.resol.est}, we provide the estimate of the resolvent. Some facts about the modified Bessel functions and technical supplements are given in Appendices \ref{appendix.bessel}, \ref{app.hom.eq.vor} and \ref{app.proof.thm.L2L2}.

{\bf Notations.} We let $C$ denote a constant and $C(a,b,c,\ldots)$ the constant depending on $a,b,c,\ldots$. Both of these may vary from line to line. We denote $\R^\ast=\{x\in\R~|~x\neq0\}$, $\R_{\ge0} = \{x\in\R~|~x \ge 0\}$, $\R_{\le0} = \{x\in\R~|~x \le 0\}$ and $\Sigma_{\phi} = \{z\in\C\setminus\{0\}~|~|\oparg z|<\phi\}$. For $z\in\C$, let $\Re z$ and $\Im z$ denote the real and imaginary parts of $z$, respectively. For $z\in\C\setminus\R_{\le0}$, let $z^\mu$ denote $e^{\mu\opLog z}$ where $z=|z|e^{i\oparg z}$, $\oparg z\in(-\pi,\pi)$ and $\opLog z = \log|z| + i\oparg z$. We take the square root $\sqrt{z}$ so that $\Re \sqrt{z}>0$. We use the function spaces 
$$
\widehat{W}^{1,2}(\Omega)
=\{p\in L^2_{{\rm loc}}(\overline{\Omega})~|~\nabla p\in L^2(\Omega)^2\}
$$
and
$$
C^\infty_{0,\sigma}(\Omega)
=\{\varphi\in C^\infty_0(\Omega)^2
~|~\opdiv\varphi=0\}, 
\qquad 
L^2_{\sigma}(\Omega) 
= 
\overline{C^\infty_{0,\sigma} (\Omega)}^{\|\,\cdot\,\|_{L^2}}. 
$$
Not to burden notation, we use the same symbols to denote the quantities for scalar-, vector- or tensor-valued functions, e.g., $\langle \cdot, \cdot\rangle$ is the inner product on $L^2(\Omega)$, $L^2(\Omega)^2$ or $L^2(\Omega)^{2\times2}$.

\section{Preliminaries}\label{sec.preliminaries}

This section collects the items used throughout the paper.

\subsection{Vectors in the polar coordinates}\label{subsec.vec.in.polar.coord.}

The polar coordinates on the exterior unit disk $\Omega$ are written as 
\begin{align*}
&x_1 = r\cos \theta, 
\qquad
x_2 = r\sin \theta, 
\quad 
r\in [1,\infty), 
\quad 
\theta\in [0,2\pi),\\
&
{\bf e}_r = \frac{x}{|x|}, 
\qquad 
{\bf e}_\theta = \frac{x^\bot }{|x|} = \partial_\theta {\bf e}_r. 
\end{align*}
Let a vector field $v=(v_1,v_2)$ on $\Omega$ be given. We set
\begin{align*}
v = v_r(r,\theta) {\bf e}_r + v_\theta(r,\theta) {\bf e}_\theta, 
\qquad
v_r = v\cdot {\bf e}_r, 
\qquad
v_\theta = v\cdot {\bf e}_\theta, 
\end{align*}
and for a given $n\in\Z$, 
\begin{align}\label{def.P_n}
\begin{split}
\mathcal{P}_n v(r,\theta)
& = v_{r,n}(r) e^{i n \theta} {\bf e}_r + v_{\theta,n}(r) e^{i n \theta} {\bf e}_\theta, \\
v_{r,n} (r) 
& := 
\frac{1}{2\pi} 
\int_0^{2\pi} 
v_r(r\cos\sigma, r\sin\sigma) e^{-in\sigma} \dd\sigma,\\
v_{\theta,n} (r) 
& := 
\frac{1}{2\pi} 
\int_0^{2\pi} 
v_\theta(r\cos\sigma, r\sin\sigma) e^{-in\sigma} \dd\sigma. 
\end{split}
\end{align}
We will use the formulas 
\begin{align}\label{formulas.polar1}
\begin{split}
|\nabla v |^2 
& = 
|\partial_r v_r|^2 + |\partial_r v_\theta|^2 
+ \frac{1}{r^2} 
(|\partial_\theta v_r - v_\theta |^2 + |v_r + \partial_\theta v_\theta |^2), \\
\opdiv v 
& = \partial_1 v_1 + \partial_2 v_2 
= \frac1r \Big(\partial_r (r v_r) + \partial_\theta v_\theta \Big), \\
\oprot v 
& = \partial_1 v_2 - \partial_2 v_1 
= \frac1r \Big(\partial_r (r v_\theta) - \partial_\theta v_r\Big), 
\end{split}
\end{align}
and
\begin{align}\label{formulas.polar2}
\begin{split}
-\Delta v 
& = \Big\{ -\partial_r \Big( \frac1r \partial_r (r v_r ) \Big)  
- \frac{1}{r^2} \partial_\theta^2 v_r 
+ \frac{2}{r^2} \partial_\theta v_\theta \Big\} {\bf e}_r \\
&\quad
+  \Big\{ - \partial_r \Big( \frac1r \partial_r (r v_\theta) \Big) 
- \frac{1}{r^2} \partial_\theta^2 v_\theta 
- \frac{2}{r^2} \partial_\theta v_r \Big\} {\bf e}_\theta.
\end{split}
\end{align}
%

\subsection{Fourier series and decomposition}\label{subsec.fourier.series.decom}

Let $n\in \Z$ and $\mathcal{P}_n$ be defined in \eqref{def.P_n}. We set, for a vector field $v = v(r,\theta)$ on $\Omega$, 
\begin{align}\label{def.v_n}
v_n(r,\theta) = \mathcal{P}_n v(r,\theta), 
\end{align}
for a scalar function $\omega=\omega(r,\theta)$ on $\Omega$, 
\begin{align}\label{def.omega_n}
\begin{split}
\mathcal{P}_n\omega(r,\theta)
&= 
\bigg( 
\frac{1}{2\pi} 
\int_0^{2\pi} 
\omega (r\cos\sigma, r\sin\sigma) e^{- in\sigma} 
\dd \sigma 
\bigg) e^{i n \theta}, \\
\omega_n(r)
&= 
(\mathcal{P}_n \omega) e^{-i n \theta}, 
\end{split}
\end{align}
and for a function space $X(\Omega)\subset L^{1}_{\rm loc}(\overline{\Omega})^2$ or $X(\Omega)\subset L^{1}_{\rm loc}(\overline{\Omega})$, 
\begin{align*}
\mathcal{P}_n X(\Omega)
=\big\{\mathcal{P}_n f ~\big|~ f\in X(\Omega)\big\}. 
\end{align*}
The definition of $f_n$ differs according to whether $f$ is vectorial or scalar. The former and latter are defined in \eqref{def.v_n} as $f_n=\mathcal{P}_n f$ and in \eqref{def.omega_n} as $f_n=(\mathcal{P}_n f) e^{-i n \theta}$, respectively.

By definition, any vector field $v\in L^2(\Omega)^2$ is expanded into the convergent series 
$$
v = \sum_{n\in\Z} \mathcal{P}_n v = \sum_{n\in\Z} v_n, 
$$
and $\mathcal{P}_n$ is an  orthogonal projection of $L^2(\Omega)^2$ onto $\mathcal{P}_n L^2(\Omega)^2$. Moreover, the following orthogonal decomposition of the subspace $L^2_\sigma(\Omega) \subset L^2(\Omega)^2$ holds: 
\begin{align}\label{def.L2.sigma.n}
L^2_\sigma(\Omega) = \bigoplus_{n\in\Z} L^2_{\sigma,n}(\Omega), 
\qquad 
L^2_{\sigma,n}(\Omega) := \mathcal{P}_n L^2_{\sigma}(\Omega). 
\end{align}

From \eqref{formulas.polar1}, we have 
\begin{align*}
\|\nabla v\|_{L^2}^2 
&= \sum_{n\in \Z} \|\nabla \mathcal{P}_n v\|_{L^2}^2, \\
|\nabla \mathcal{P}_n v |^2 
&= |\partial_r v_{r,n}|^2 
+ |\partial _r v_{\theta,n}|^2 
+ \frac{1+n^2}{r^2} (|v_{r,n}|^2 + |v_{\theta,n}|^2)
- \frac{4 n}{r^2} \Im( v_{\theta,n} \overline{v_{r,n}}). 
\end{align*}
In particular, 
$$
|\partial_r v_{r,n}|^2 
+ |\partial _r v_{\theta,n}|^2 
+ \frac{(|n|-1)^2}{r^2} (|v_{r,n}|^2 + |v_{\theta,n}|^2)
\le |\nabla \mathcal{P}_n v|^2. 
$$
Therefore, if $|n|\neq1$, the Hardy-type inequality
$$
\Big\|
(r,\theta) \mapsto \frac{\mathcal{P}_n v(r,\theta)}{r}\Big\|_{L^2} 
\le \|\nabla \mathcal{P}_n v\|_{L^2} 
$$
holds. Thus it is convenient to set 
$$
v_{\neq} = v - \sum_{|n|=1} \mathcal{P}_n v, 
$$
for which we have 
\begin{align}\label{Hardy.ineq.nonzero}
\Big\|
(r,\theta) \mapsto \frac{v_{\neq}(r,\theta)}{r}\Big\|_{L^2} 
\le \|\nabla v_{\neq}\|_{L^2}. 
\end{align}

Again from \eqref{formulas.polar1}, we have, for $v\in W^{1,2}(\Omega)^2$, 
$$
\mathcal{P}_n
\opdiv v = \opdiv \mathcal{P}_n v, 
\qquad
\mathcal{P}_n \oprot v = \oprot \mathcal{P}_n v 
$$
and, from \eqref{formulas.polar2}, for $v\in W^{2,2}(\Omega)^2$, 
$$
\mathcal{P}_n \Delta v = \Delta \mathcal{P}_n v. 
$$
Since the condition ${\bf e}_r\cdot v=0$ on $\partial\Omega$ is preserved under $\mathcal{P}_n$, it can be shown that 
\begin{align*}
\mathcal{P}_n \mathbb{P} 
= \mathbb{P} \mathcal{P}_n,
\qquad
L^2_{\sigma,n}(\Omega) 
= \overline{\mathcal{P}_n C^\infty_{0,\sigma}(\Omega)}^{\|\,\cdot\,\|_{L^2}}. 
\end{align*}
We refer to Farwig and Neustupa \cite[Lemma 3.1]{Farwig-Neustupa(2007)} for a more detailed proof. Although the proof in \cite{Farwig-Neustupa(2007)} is for the three-dimensional cases, a similar argument is applicable.

Now we define the closed linear operator $\mathbb{A}_{n}$ on $L^2_{\sigma,n}(\Omega)$ in \eqref{def.L2.sigma.n} by 
$$
\mathbb{A}_{n} = \mathbb{A}|_{L^2_{\sigma,n}(\Omega) \cap D(\mathbb{A})},
\qquad D(\mathbb{A}_{n}) = L^2_{\sigma,n}(\Omega) \cap D(\mathbb{A}). 
$$
It is not hard to see that $\mathbb{A}_{n}$ is nonnegative and self-adjoint. Also, keeping the relation 
$$
\mathcal{P}_n {\mathbb P} V^\bot \oprot v 
= {\mathbb P} V^\bot \oprot v_n, \quad v\in W^{1,2}(\Omega)^2 
$$
in mind, we define the closed linear operator $\mathbb{A}_{V, n}$ on $L^2_{\sigma,n}(\Omega)$ by 
$$
\mathbb{A}_{V,n} = \mathbb{A}_{V}|_{L^2_{\sigma,n}(\Omega) \cap D(\mathbb{A}_V)}, 
\qquad D(\mathbb{A}_{V,n}) = D(\mathbb{A}_{n}). 
$$

\subsection{Equations in the polar coordinates}\label{subsec.eqs.polar.coord}

To study the operator $\mathbb{A}_{V,n}$, we consider 
\begin{equation}\tag{R$_n$}\label{eq.resol.n}
(\lambda + \mathbb{A}_{V,n}) v_n = f_n 
\end{equation}
for given $\lambda\in\C\setminus\R_{\le0}$ and $f_n\in L^2_{\sigma,n}(\Omega)$. The equation is equivalent to the system 
\begin{equation}\label{eq.resol.n.sys}
\left\{
\begin{array}{ll}
\lambda v_n - \Delta v_n + V^\bot \oprot v_n
+ \nabla q_n
=f_n &\mbox{in}\ \Omega \\
\opdiv v_n = 0 &\mbox{in}\ \Omega \\
v_n = 0 &\mbox{on}\ \partial\Omega, 
\end{array}\right.
\end{equation}
with some pressure $\nabla q_n$. Operating $\oprot$ to the first line, we see that $\oprot v_n$ solves 
\begin{equation}\label{eq.resol.n.sys.vor}
\lambda (\oprot v_n) 
- \Delta (\oprot v_n)
+ V\cdot\nabla(\oprot v_n)
= \oprot f_n. 
\end{equation}
In the polar coordinates on $\Omega$ where $v_n=v_n(r,\theta)$ is written as 
$$
v_n(r,\theta) = v_{r,n}(r) e^{i n \theta} {\bf e}_r + v_{\theta,n}(r) e^{i n \theta} {\bf e}_\theta, 
$$
we see from \eqref{eq.resol.n.sys} that $(v_{r,n}(r),v_{\theta,n}(r))$ and $q_n(r)$ satisfy 
\begin{align}
&\begin{aligned}\label{eq.polar.vr}
&
\lambda v_{r,n}
-\frac{\dd}{\dd r} \Big(\frac1r\frac{\dd}{\dd r} (r v_{r,n})\Big)
+ \frac{n^2}{r^2} v_{r,n}
+ \frac{2in}{r^2} v_{\theta,n} \\
&\qquad\qquad
- \frac{\alpha}{r^2} \Big(\frac{\dd}{\dd r} (r v_{\theta,n}) - in v_{r,n}\Big)
+ \frac{\dd q_n}{\dd r} 
= f_{r,n}, \quad r>1, \\
\end{aligned} \\
&\begin{aligned}\label{eq.polar.vtheta}
&
\lambda v_{\theta,n}
-\frac{\dd}{\dd r} \Big(\frac1r\frac{\dd}{\dd r} (r v_{\theta,n})\Big) 
+ \frac{n^2}{r^2} v_{\theta,n}
- \frac{2in}{r^2} v_{r,n} \\
&\qquad\qquad
- \frac{\delta}{r^2} \Big(\frac{\dd}{\dd r} (r v_{\theta,n}) - in v_{r,n}\Big)
+ \frac{i n}{r} q_n 
= f_{\theta,n}, \quad r>1 
\end{aligned}
\end{align}
and the divergence-free and the no-slip boundary conditions 
\begin{align}\label{eq.polar.div-free.no-slip}
\frac{\dd}{\dd r}(r v_{r,n}) + in v_{\theta,n}=0, 
\quad 
r>1, 
\qquad 
v_{r,n}(1)=v_{\theta,n}(1)=0. 
\end{align}
Moreover, from \eqref{eq.resol.n.sys.vor}, $\omega_n(r):=(\oprot v_n)_n(r)$ satisfies 
\begin{align}\label{eq.polar.vor}
- \frac{\dd^2 \omega_n}{\dd r^2} 
- \frac{1+\delta}{r} \frac{\dd \omega_n}{\dd r} 
+ \Big(\lambda + \frac{n^2+i\alpha n}{r^2} \Big) \omega_n 
= (\oprot f_n)_n, 
\quad 
r>1. 
\end{align}
%

\subsection{Biot-Savart law}\label{subsec.biot-savart}

To simplify the explanation, only in this subsection, we use the function space 
$$
L^\infty_s(\Omega) 
= \{ f\in L^\infty(\Omega)~|~ \|f\|_{L^\infty_s} <\infty \},
\qquad 
\|f\|_{L^\infty_s} := \esssup_{x\in\Omega}\, |x|^s |f(x)|. 
$$
For a given $\omega\in L^\infty_2(\Omega)$, we consider the Poisson equation
\begin{equation*}
\left\{
\begin{array}{ll}
-\Delta \psi = \omega&\mbox{in}\ \Omega \\
\psi=0&\mbox{on}\ \partial\Omega. 
\end{array}\right.
\end{equation*}
Let $\omega\in\mathcal{P}_n L^\infty_2(\Omega)$ with $|n|\ge 1$ and let $\psi$ be the decaying solution, called the streamfunction. Applying the notation in \eqref{def.omega_n}, we find that $\psi_n=\psi_n(r)$ satisfies 
\begin{align}\label{eq.streamfunc.}
-\frac{\dd^2 \psi_n}{\dd r^2}  
- \frac{1}{r} \frac{\dd \psi_n}{\dd r} 
+ \frac{n^2}{r^2} \psi_n 
= \omega_n,  \quad r>1,  
\qquad 
\psi_n (1) =0.
\end{align}
By elementary computation, we see that $\psi_n= \psi_n[\omega_n]$ is given by 
\begin{align}\label{def.psi_n}
\begin{split}
\psi_n[\omega_n] (r) 
&= \frac{1}{2 |n|} 
\bigg(-d_n [\omega_n] r^{-|n|} \\
&\qquad\qquad
+ r^{-|n|} \int_1^r s^{|n|+1} \omega_n (s) \dd s   
+ r^{|n|}\int_r^\infty s^{-|n|+1} \omega_n (s) \dd s \bigg ), \\
d_n [\omega_n] 
&:= \int_1^\infty s^{-|n|+1} \omega_n (s) \dd s. 
\end{split}
\end{align}
The following vector field 
\begin{align}\label{def.Biot-Savart}
\begin{split}
&\mathcal{V}_n [\omega_n](r,\theta) 
= \mathcal{V}_{r,n}[\omega_n](r) e^{i n \theta}{\bf e}_r  
+  \mathcal{V}_{\theta,n}[\omega_n](r) e^{i n\theta} {\bf e}_\theta, \\
& \mathcal{V}_{r,n} [\omega_n](r)
:= \frac{i n}{r} \psi_n [\omega_n](r), \qquad 
\mathcal{V}_{\theta,n} [\omega_n](r)
:=-\frac{\dd}{\dd r} \psi_n [\omega_n](r)
\end{split}
\end{align}
is called the Biot-Savart law. It is straightforward to see that
\begin{align}\label{properties.Biot-Savart}
\begin{split}
&\opdiv \mathcal{V}_n[\omega_n] 
= 0, 
\qquad
\oprot \mathcal{V}_n[\omega_n](r,\theta) 
= \omega_n(r) e^{in\theta}, 
\qquad 
({\bf e}_r \cdot \mathcal{V}_n [\omega_n])|_{\partial\Omega} = 0. 
\end{split}
\end{align}
If additionally $\omega\in L^\infty_\rho(\Omega)$ with some $\rho>2$, we can check that $\mathcal{V}_n [\omega_n]\in W^{1,2}(\Omega)^2$.

Here are useful two propositions in the subsequent sections. The reader is referred to \cite[Proposition 2.6 and Lemma 3.1]{Maekawa(2017a)}, \cite[Proposition 2.1]{Higaki(2022)} for the proof of the first proposition and \cite[Corollary 2.7]{Maekawa(2017a)}, \cite[Proposition 2.2]{Higaki(2022)} for the proof of the second.
%
\begin{proposition}\label{prop.biot-savart}
Let $|n|\ge1$ and $v_n\in \mathcal{P}_n W^{1,2}_0(\Omega)^2$. Set $\omega_n=(\oprot v_n)_n$. If $\opdiv v_n=0$ and $\omega_n\in L^\infty_\rho(\Omega)$ for some $\rho>2$, we have $v_n=\mathcal{V}_n[\omega_n]$ and $d_n[\omega_n]=0$ in \eqref{def.psi_n}. 
\end{proposition}
%

%
\begin{proposition}\label{prop.rot.zero}
Let $|n|\ge1$ and $f_n\in \mathcal{P}_n L^{2}(\Omega)^2$. If $\oprot f_n = 0$ in the sense of distributions, we have $f=\nabla \mathcal{P}_n p$ for some $\mathcal{P}_n p\in \mathcal{P}_n \widehat{W}^{1,2}(\Omega)$. 
\end{proposition}
%

\section{Spectral analysis}\label{sec.spec.anal.}

In this section, we study the spectrum of the operator $-{\mathbb A}_V$. The main result is Proposition \ref{prop.charact.disc.ev} which characterizes the discrete spectrum $\sigma_{{\rm disc}}(-{\mathbb A}_V)$ as zeros of certain analytic functions. We are aware that the presentation in Subsections \ref{subsec.notation.spec.anal.} and \ref{subsec.perturb.thr} has similarity to \cite{Maekawa(2017a)} treating the case $\delta=0$. This is quite natural because, in analysis in the $L^2$-framework, especially in computation of the numerical ranges, one can control terms involving $\delta W$ by using the Hardy-type inequality \eqref{Hardytype.ineq.W}. Consequently, for example, the statement of Proposition \ref{prop.spec.small} holds independently of sufficiently small $\delta$. However, the difference from the case $\delta=0$ appears when one studies the spectrum of $-{\mathbb A}_V$. Indeed, in Proposition \ref{prop.charact.disc.ev}, the functions characterizing the discrete spectrum depend both on $\alpha$ and $\delta$. These functions will be studied in detail quantitatively in the next section.

\subsection{Notation}\label{subsec.notation.spec.anal.}

Let us recall the standard notation in the perturbation theory. Our main reference is Kato \cite{Kato(1976)}. Let $X$ be a Banach space and $\mathbb{L}: D(\mathbb{L})\subset X \to X$ be a closed linear operator. We let $N(\mathbb{L})$ denote the null space of $\mathbb{L}$, $R(\mathbb{L})$ its range, and $X/R(\mathbb{L})$ the quotient space of $X$ by $R(\mathbb{L})$. Moreover, $\rho(\mathbb{L})$ denotes the resolvent set of $\mathbb{L}$, $\sigma(\mathbb{L})$ its spectrum, and $\sigma_{{\rm disc}}(\mathbb{L})$ its discrete spectrum, namely, the set of isolated eigenvalues of $\mathbb{L}$ with finite multiplicity. The operator $\mathbb{L}$ is said to be semi-Fredholm if $R(\mathbb{L})$ is closed and at least one of $\operatorname{dim} N(\mathbb{L})$ or $\operatorname{dim} X/R(\mathbb{L})$ is finite. If $\mathbb{L}$ is semi-Fredholm, the index of $\mathbb{L}$ 
$$
\opind(\mathbb{L}) = \opdim N(\mathbb{L}) - \opdim X/R(\mathbb{L}) 
$$
is well-defined, taking values in $[-\infty,\infty]$. Finally, let us set 
$$
\rho_{{\rm sf}}(\mathbb{L})
= \{\lambda\in\C~|~\text{$\lambda-\mathbb{L}$ is semi-Fredholm}\}, 
\qquad
\sigma_{{\rm ess}}(\mathbb{L}) 
= \C\setminus\rho_{{\rm sf}}(\mathbb{L})
$$
and call the semi-Fredholm domain of $\mathbb{L}$ and the essential spectrum of $\mathbb{L}$, respectively.

Generally, $\rho_{{\rm sf}}(\mathbb{L})$ is the union of a countable (at most) family of connected open sets. From the argument in \cite[Chapter $\mathrm{I}\hspace{-1.2pt}\mathrm{V}$ \S5 6]{Kato(1976)}, we see that $\operatorname{ind}(\lambda-\mathbb{L})$ is a constant function of $\lambda$ in each component $G$ of $\rho_{{\rm sf}}(\mathbb{L})$. Moreover, both $\operatorname{dim} N(\lambda-\mathbb{L})$ and $\operatorname{dim} X/R(\lambda-\mathbb{L})$ are constants in each $G$ except for an isolated set of values of $\lambda$. Therefore, if these constants are zero in particular, then $G$ is contained in $\rho(\mathbb{L})$ with possible exception of isolated points of $\sigma(\mathbb{L})$, which are, isolated eigenvalues of finite algebraic multiplicity.

\subsection{Perturbation theory}\label{subsec.perturb.thr}

We start with the perturbation theory of operators. 
%
\begin{proposition}\label{prop.spec.general}
Let $\alpha,\delta\in\R$. We have the following. 
\begin{enumerate}[(1)]
\item\label{item1.prop.spec.general}
$\sigma_{{\rm ess}}(-{\mathbb A}_V) = \R_{\le0}$ and 
$\sigma_{{\rm disc}}(-{\mathbb A}_V) \sqcup \rho(-\mathbb{A}_V) = \C\setminus\R_{\le0}$.

\item\label{item2.prop.spec.general}
The same statement with $\mathbb{A}_V$ replaced by $\mathbb{A}_{V,n}$ holds for $n\in\Z$.

\item\label{item3.prop.spec.general}
$\sigma_{{\rm disc}}(-{\mathbb A}_{V}) = \bigcup_{n\in\Z} \sigma_{{\rm disc}}(-{\mathbb A}_{V,n})$ and $\rho(-{\mathbb A}_{V})=\bigcap_{n\in\Z} \rho(-{\mathbb A}_{V,n})$. 
\end{enumerate}
\end{proposition}
%
%
\begin{proof}
(1) The fact that $\sigma(-{\mathbb A}) = \R_{\le0}$ is well-known and essentially due to Ladyzhenskaya \cite{Ladyzhenskaya(1969)}. Based on this fact, one can prove that $\sigma_{{\rm ess}}(-{\mathbb A}) = \sigma(-{\mathbb A})$ by showing the non-existence of eigenvalues in a similar manner as in \cite[Lemma 2.6]{Farwig-Neustupa(2007)}, or by using the property of the index $\opind(\lambda+\mathbb{A})$ as is done in \cite[Proof of Proposition 2.12]{Maekawa(2017a)}. Because of the regularity and decay of $V$, the operator $-{\mathbb A}_V + {\mathbb A} = -{\mathbb P} V^\bot \oprot $ is relatively compact with respect to $- {\mathbb A}$. The proof is quite similar to the one in \cite[Section 2.4]{Maekawa(2017a)} for the case $\delta=0$ and thus we omit the details. Hence, from \cite[Chapter $\mathrm{I}\hspace{-1.2pt}\mathrm{V}$, Theorem 5.35]{Kato(1976)}, we see that $-{\mathbb A}_V$ and $-{\mathbb A}$ have the same essential spectrum. This implies the first statement.

For the second statement, we first observe that the equality 
$$
\opind(\lambda + \mathbb{A}_V) 
= \opind(\lambda + \mathbb{A})=0, 
\quad 
\lambda\in \rho_{{\rm sf}}(-\mathbb{A}_V)=\C\setminus\R_{\le0}
$$
holds by \cite[Chapter $\mathrm{I}\hspace{-1.2pt}\mathrm{V}$, Theorems 5.26 and 5.35]{Kato(1976)}. Hence, since $\C\setminus\R_{\le0}$ has only one component, by the argument in Subsection \ref{subsec.notation.spec.anal.}, we only need to prove that $\opdim N(\lambda + \mathbb{A}_V)=0$ for at least one point $\lambda\in \C\setminus\R_{\le0}$. For this purpose, we consider 
$$
\Theta(-\mathbb{A}_{V}) 
= \{\langle -\mathbb{A}_{V} v,v\rangle
~|~\text{$v\in D(\mathbb{A}_{V})$ with $\|v\|_{L^2}=1$}\}, 
$$
which is called the numerical range of $-\mathbb{A}_{V}$; see \cite[Chapter $\mathrm{V}$ \S3 2]{Kato(1976)}.

Let $v\in D(\mathbb{A}_{V})$. From the relation 
$$
\langle -\mathbb{A}_{V} v, v\rangle
= -\|\nabla v\|_{L^2}^2 
- \langle V^\bot \oprot v, v\rangle, 
$$
we have 
\begin{align}\label{est1.proof.prop.spec.general}
|\Im \langle -\mathbb{A}_{V} v, v\rangle| 
+ \Re \langle -\mathbb{A}_{V} v, v\rangle 
\le
-\|\nabla v\|_{L^2}^2 
+ 2|\langle V^\bot \oprot v, v\rangle|. 
\end{align}
Now let $v\in D(\mathbb{A}_{V})$ and $\|v\|_{L^2}=1$. The term $2|\langle V^\bot \oprot v, v\rangle|$ is estimated as 
\begin{align*}
2|\langle V^\bot \oprot v, v\rangle|
&\le 
2 (|\alpha| + |\delta|) \|v\|_{L^2} \|\oprot v\|_{L^2} \\
&=
2 (|\alpha| + |\delta|) \|\nabla v\|_{L^2} \\
&\le
(|\alpha| + |\delta|)^2 + \|\nabla v\|_{L^2}^2. 
\end{align*}
We have used $\|\oprot v\|_{L^2}=\|\nabla v\|_{L^2}$ for $v\in W^{1,2}_0(\Omega)^2\cap L^2_\sigma(\Omega)$ in the second line and the Young inequality in the third line. Hence we obtain 
$$
|\Im \langle -\mathbb{A}_{V} v, v\rangle| 
+ \Re \langle -\mathbb{A}_{V} v, v\rangle 
- (|\alpha| + |\delta|)^2
\le
0, 
$$
which leads to the inclusion 
$$
\overline{\Theta(-{\mathbb A}_{V})} 
\subset 
\{
\lambda\in\C
~|~
|\Im \lambda| + \Re \lambda - (|\alpha| + |\delta|)^2 \le0
\}.  
$$
From \cite[Chapter $\mathrm{V}$, Theorem 3.2]{Kato(1976)}, we know that $\opdim N(\lambda + \mathbb{A}_V)=0$ for any $\lambda$ belonging to the complement of the right-hand side 
$$
\{
\lambda\in\C
~|~
|\Im \lambda| + \Re \lambda - (|\alpha| + |\delta|)^2 >0
\}.  
$$
This set is obviously a subset of $\C\setminus\R_{\le0}$ and thus the second statement follows.

(2) The fact that $\sigma(-\mathbb{A}_{n}) = \R_{\le0}$ can be proved in a similar manner as in \cite[Lemma 3.3]{Farwig-Neustupa(2007)}, and $\sigma_{{\rm ess}}(-{\mathbb A}_{n}) = \sigma(-{\mathbb A}_{n})$ follows by the property of $\opind(\lambda+\mathbb{A}_{n})$. Hence the first statement $\sigma_{{\rm ess}}(-{\mathbb A}_{n}) = \R_{\le0}$ follows from the relative compactness of $-{\mathbb A}_{V,n} + {\mathbb A}_n$ with respect to $-{\mathbb A}_n$. The second statement $\sigma_{{\rm disc}}(-{\mathbb A}_{V,n}) \sqcup \rho(-\mathbb{A}_{V,n}) = \C\setminus\R_{\le0}$ can be deduced from the same discussion as above with $\mathbb{A}_V$ replaced by $\mathbb{A}_{V,n}$.

(3) It suffices to prove the first statement $\sigma_{{\rm disc}}(-{\mathbb A}_{V})=\bigcup_{n\in\Z} \sigma_{{\rm disc}}(-{\mathbb A}_{V,n})$. If $\lambda\in\sigma_{{\rm disc}}(-{\mathbb A}_V)$, there is a nonzero $v\in D({\mathbb A}_V)$ such that $(\lambda+{\mathbb A}_V)v=0$. Choosing $n\in\Z$ such that $v_n=\mathcal{P}_nv\neq0$, we have $v_n\in D({\mathbb A}_{V,n})$ and $(\lambda+{\mathbb A}_{V,n})v_n=0$. Then we see that $\lambda\in \sigma_{{\rm disc}}(-{\mathbb A}_{V,n})$ and hence $\lambda\in\bigcup_{n\in\Z} \sigma_{{\rm disc}}(-{\mathbb A}_{V,n})$. Oppositely, if $\lambda\in\bigcup_{n\in\Z} \sigma_{{\rm disc}}(-{\mathbb A}_{V,n})$, then there are $n\in\Z$ and nonzero $v_n\in D({\mathbb A}_{V,n})$ such that $(\lambda+{\mathbb A}_{V,n})v_n=0$. Then we have $v_n\in D({\mathbb A}_{V})$ and $(\lambda+{\mathbb A}_{V})v_n=0$ and hence $\lambda\in \sigma_{{\rm disc}}(-{\mathbb A}_{V})$. The proof is complete. 
\end{proof}
%

The estimate of the numerical range $\Theta(-\mathbb{A}_{V})$ in the proof of Proposition \ref{prop.spec.general} is quite rough. We consider its refinement in Lemma \ref{lem.est.innerprod.1} to prove Proposition \ref{prop.spec.small} below.

%
\begin{lemma}\label{lem.est.innerprod.1}
Let $\alpha,\delta\in\R$. For $v\in W^{1,2}_0(\Omega)^2\cap L^2_\sigma(\Omega)$, we have 
\begin{align}\label{est1.lem.est.innerprod}
|\langle V^\bot \oprot v, v\rangle|
\le 
|\alpha| 
\Big|
\sum_{|n|=1} 
\langle U^\bot \oprot v_n, v_n\rangle
\Big| 
+ |\alpha| \|\nabla v_{\neq}\|_{L^2}^2
+ |\delta| \|\nabla v\|_{L^2}^2. 
\end{align}
Moreover, for any $T>0$, 
\begin{align}\label{est2.lem.est.innerprod}
\Big|\sum_{|n|=1} 
\langle U^\bot \oprot v_n, v_n\rangle
\Big| 
\le 
2 h(T) 
\Big(
\sum_{|n|=1} \|\nabla v_n\|_{L^2}^2 
\Big)
+ \frac{1}{4T^2 h(T)} 
\Big(
\sum_{|n|=1} \|v_n\|_{L^2}^2 
\Big). 
\end{align}
Here the function $h=h(T)$ is defined by 
$$
h(T) 
= \int_0^T \frac{1}{\tau} e^{-\frac{1}{\tau}} \dd \tau, \quad T>0, 
\label{ThetaT}
$$
which satisfies 
\begin{align}\label{est3.lem.est.innerprod}
e^{-1}\log T
\le h(T) 
\le \log T, \quad T>e. 
\end{align}
\end{lemma}
%
%
\begin{proof}
By the definition of $V$, we see that 
$$
\langle V^\bot \oprot v, v\rangle 
= 
\alpha \langle U^\bot \oprot v, v\rangle 
- \delta \langle W^\bot \oprot v, v\rangle. 
$$
The Fourier series expansion leads to 
$$
\langle U^\bot \oprot v, v\rangle
= 
\sum_{|n|=1} 
\langle U^\bot \oprot v_n, v_n\rangle
+ \langle U^\bot \oprot v_{\neq}, v_{\neq}\rangle. 
$$
We have
\begin{align*}
|\langle U^\bot \oprot v_{\neq}, v_{\neq}\rangle|
&\le 
\Big\|
(r,\theta) 
\mapsto 
\frac{v_{\neq}(r,\theta)}{r}
\Big\|_{L^2} 
\|\oprot v_{\neq}\|_{L^2} \\
&\le 
\|\nabla v_{\neq}\|_{L^2}^2, 
\end{align*}
where the Hardy-type inequality \eqref{Hardy.ineq.nonzero} and $\|\oprot u\|_{L^2}=\|\nabla u\|_{L^2}$ for $u\in W^{1,2}_0(\Omega)^2\cap L^2_\sigma(\Omega)$ are applied. Also, from \eqref{eq.bilinear.rot} and \eqref{Hardytype.ineq.W} in the introduction, 
$$
|\langle W^\bot \oprot v, v\rangle|
\le \|\nabla v\|_{L^2}^2. 
$$
Combining all the estimates so far, we obtain \eqref{est1.lem.est.innerprod}.

Next let $|n|=1$. We compute 
$$
|\langle U^\bot \oprot v_n, v_n\rangle|
\le 
\int_{0}^{2\pi} \int_{1}^{\infty}
\frac1r |(\oprot v_n)_n(r)| |v_{r,n}(r)| r\dd r \dd \theta. 
$$
As is shown in \cite[Proof of Lemma 3.26]{Maekawa(2017a)}, we have 
\begin{align*}
&\int_{0}^{2\pi} \int_{1}^{\infty}
\frac1r |(\oprot v_n)_n(r)| |v_{r,n}(r)| r\dd r \dd \theta \\
&\le
h(T) \|\nabla v_n\|_{L^2}^2 
+ \frac{1}{T} \|v_n\|_{L^2}\ \|\nabla v_n\|_{L^2}, 
\quad T>0. 
\end{align*}
The Young inequality yields
$$
\frac{1}{T} 
\|v_n\|_{L^2}\ \|\nabla v_n\|_{L^2} 
\le
h(T) \|\nabla v_n\|_{L^2}^2 + \frac{1}{4T^2 h(T)} \|v_n\|_{L^2}^2. 
$$
These estimates imply \eqref{est2.lem.est.innerprod} after summation. The proof is complete.
\end{proof}
%

%
\begin{proposition}\label{prop.spec.small}
Let $\alpha,\delta\in\R$ be sufficiently small. We have the following. 
\begin{enumerate}[(1)]
\item\label{item1.prop.spec.small}
The set 
$$
\Sigma_{\frac34 \pi} + 4e\alpha^2 e^{-\frac{1}{4|\alpha|}}
= 
\Big\{\lambda\in\C~\Big|~|\Im \lambda| + \Re \lambda - 4e\alpha^2 e^{-\frac{1}{4|\alpha|}} > 0 \Big\}
$$
is contained in $\rho(-\mathbb{A}_{V})$.

\item\label{item2.prop.spec.small} 
The same statement with $\mathbb{A}_V$ replaced by $\mathbb{A}_{V,n}$ holds for each $|n|=1$.

\item\label{item3.prop.spec.small}
The set $\Sigma_{\frac{3}{4}\pi}$ is contained in $\rho(-\mathbb{A}_{V,n})$ for each $|n|\neq1$. 
\end{enumerate}
\end{proposition}
%
%
\begin{proof}
Let us consider the numerical range as in the proof of Proposition \ref{prop.spec.general}.

(1) We first estimate $2|\langle V^\bot \oprot v, v\rangle|$. Let $v\in D(\mathbb{A}_{V})$. Using Lemma \ref{lem.est.innerprod.1}, we have 
\begin{align*}
2|\langle V^\bot \oprot v, v\rangle|
&\le 
2|\alpha| 
\bigg(
2 h(T) \|\nabla v\|_{L^2}^2 
+ \frac{1}{4T^2 h(T)} \| v\|_{L^2}^2
\bigg) \\
&\quad 
+ 2(|\alpha|+|\delta|) \|\nabla v\|_{L^2}^2 \\
&\le 
2(2|\alpha| h(T) + |\alpha| + |\delta|) 
\|\nabla v\|_{L^2}^2
+ \frac{|\alpha|}{2T^2 h(T)} 
\| v\|_{L^2}^2. 
\end{align*}
Let us choose $T = e^{\frac1{8|\alpha|}}$. From \eqref{est3.lem.est.innerprod}, we see that 
$$
2(2|\alpha| h(T) + |\alpha| + |\delta|)
\le 
\frac12 + 2(|\alpha|+|\delta|) 
$$
and that 
$$
\frac{|\alpha|}{2T^2 h(T)}
\le \frac{e |\alpha|}{2T^2 \log T}
= 4e\alpha^2 e^{-\frac{1}{4|\alpha|}}. 
$$
Hence we obtain 
\begin{align}\label{est1.proof.prop.spec.small}
2|\langle V^\bot \oprot v, v\rangle|
\le 
\Big\{
\frac12 + 2(|\alpha|+|\delta|)
\Big\}
\|\nabla v\|_{L^2}^2
+ 4e\alpha^2 e^{-\frac{1}{4|\alpha|}} 
\| v\|_{L^2}^2. 
\end{align}
Note that up to this point the smallness of  $\alpha,\delta$ is not needed.

Now let $v\in D(\mathbb{A}_{V})$ and $\|v\|_{L^2}=1$. From \eqref{est1.proof.prop.spec.general} and \eqref{est1.proof.prop.spec.small}, we have 
\begin{align*}
\begin{split}
&|\Im \langle -\mathbb{A}_{V} v, v\rangle|
+ \Re \langle -\mathbb{A}_{V} v, v\rangle 
- 4e\alpha^2 e^{-\frac{1}{4|\alpha|}} \\
&\le
\Big\{-\frac12 + 2(|\alpha|+|\delta|) \Big\}
\|\nabla v\|_{L^2}^2. 
\end{split}
\end{align*}
Therefore, for sufficiently small $\alpha,\delta$, we obtain the inclusion 
$$
\overline{\Theta(-{\mathbb A}_{V})} 
\subset 
\{\lambda\in\C
~|~
|\Im \lambda| + \Re \lambda - 4e\alpha^2 e^{-\frac{1}{4|\alpha|}}\le0
\}.  
$$
Then the statement follows from the same argument as in the proof of Proposition \ref{prop.spec.general} (\ref{item2.prop.spec.general}).

(2) A similar proof as above leads to the statement.

(3) Let $v\in D(\mathbb{A}_{V,n})$ with $\|v\|_{L^2}=1$. Using Lemma \ref{lem.est.innerprod.1}, we estimate 
\begin{align*}
|\Im \langle -\mathbb{A}_{V,n} v, v\rangle| 
+ \Re \langle -\mathbb{A}_{V,n} v, v\rangle 
&\le
-\|\nabla v\|_{L^2}^2 
+ 2|\langle V^\bot \oprot v, v\rangle| \\
&\le
\{-1 + 2(|\alpha|+|\delta|)\} 
\|\nabla v\|_{L^2}^2. 
\end{align*}
Thus the statement follows. The proof of Proposition \ref{prop.spec.small} is complete. 
\end{proof}
%

\subsection{Analysis by explicit computation}\label{subsec.anal.expct.compt}

Proposition \ref{prop.spec.small} does not provide information on the discrete spectrum of $-{\mathbb A}_{V}$ near the origin. This is a consequence of the fact that the Hardy inequality fails to hold in two-dimensional exterior domains. Therefore, we investigate the homogeneous equation of \eqref{eq.resol.n} by a more explicit computation, exploiting the symmetry of the exterior disk $\Omega$.

For $|n|\ge1$, we define 
\begin{align}\label{def.xi}
\xi_n 
= 
\xi_n(\alpha,\delta)
= 
\bigg[
\Big\{n^2 + \Big(\frac{\delta}{2}\Big)^2\Big\}^\frac12
+ i\alpha n \bigg]^\frac12 
\end{align}
and
\begin{align}\label{def.Fn}
F_n(\sqrt{\lambda}) 
= F_n(\sqrt{\lambda};\alpha,\delta) 
= \int_1^\infty 
s^{-|n|+1-\frac{\delta}2} 
K_{\xi_n}(\sqrt{\lambda} s) \dd s, 
\quad 
\lambda\in\C\setminus\R_{\le0}. 
\end{align}
%

%
\begin{proposition}\label{prop.charact.disc.ev}
Let $\alpha,\delta\in\R$. We have the following. 
\begin{enumerate}[(1)]
\item\label{item1.prop.charact.disc.ev}
$\sigma_{{\rm disc}}(-{\mathbb A}_{V,0}) = \emptyset$.

\item\label{item2.prop.charact.disc.ev}
$\sigma_{{\rm disc}}(-{\mathbb A}_{V,n}) 
= \{\lambda\in\C\setminus\R_{\le0}~|~F_n(\sqrt{\lambda})=0\}$ for $|n|\ge1$. 
\end{enumerate}
\end{proposition}
%
%
\begin{proof}
(1) Let $\lambda\in\C\setminus\R_{\le0}$. In view of Proposition \ref{prop.spec.general} (\ref{item2.prop.spec.general}) and $\opind(\lambda+\mathbb{A}_{V,0})=0$, we will show that the equation $(\lambda + {\mathbb A}_{V,0}) v_0 = 0$ has only the trivial solution in $D({\mathbb A}_{V,0})$. Put $n=0$ in \eqref{eq.polar.vr}--\eqref{eq.polar.div-free.no-slip} with $f_0=0$. The conditions in \eqref{eq.polar.div-free.no-slip} imply that $v_{r,0}(r)=0$ and hence that $v_{0}=v_{\theta,0}(r) {\bf e}_\theta$. 
From \eqref{eq.polar.vtheta}--\eqref{eq.polar.div-free.no-slip}, we see that $v_{\theta,0}(r)$ satisfies 
$$
-\frac{\dd^2 v_{\theta,0}}{\dd r^2} 
-\frac{1+\delta}r \frac{\dd v_{\theta,0}}{\dd r}  
+\Big(
\lambda + \frac{1-\delta}{r^2} 
\Big) v_{\theta,0}
= 0, \quad r>1, \qquad v_{\theta,0}(1)=0.
$$
By summability, the solution is given by, with some constant $c_0$, 
$$
v_{\theta,0}(r) = c_0 r^{-\frac{\delta}2} K_{|1-\frac{\delta}2|}(\sqrt{\lambda} r). 
$$
Then the boundary condition leads to $c_0=0$ since $K_{\nu}(\cdot)$ has no zeros in $\Sigma_{\frac{\pi}{2}}$ if $\nu\ge0$; see \cite[Chapter $\mathrm{X}\hspace{-1.2pt}\mathrm{V}$ 15$\cdot$7]{Watson(1944)}. Hence we obtain that $v_0=v_{\theta,0}(r) {\bf e}_\theta = 0$, which is to be shown.

(2) Let $\lambda\in\C\setminus\R_{\le0}$. In view of Proposition \ref{prop.spec.general} (\ref{item2.prop.spec.general}) and $\opind(\lambda+\mathbb{A}_{V,n})=0$, we will show that the equation $(\lambda + {\mathbb A}_{V,n}) v_n = 0$ admits a nontrivial solution in $D({\mathbb A}_{V,n})$ if and only if $F_n(\sqrt{\lambda})=0$. Let $v_n\in D({\mathbb A}_{V,n})$ be nontrivial and solve $(\lambda + {\mathbb A}_{V,n}) v_n = 0$. Notice that $v_n$ is smooth by the elliptic regularity of the Stokes system. Setting $\omega_n(r)=(\oprot v_n)_n(r)$, we see that $\omega_n$ satisfies the homogeneous equation of \eqref{eq.polar.vor}. Its linearly independent solutions are \eqref{sols.eta.K.I} in Appendix \ref{app.hom.eq.vor}. By the summability of $v_n$, we must have, with some constant $c_n$, 
$$
\omega_n(r) 
= c_n r^{-\frac{\delta}2} K_{\zeta_n}(\sqrt{\lambda} r). 
$$
Since $\omega_n(r)$ decays exponentially as $r\to\infty$, Proposition \ref{prop.biot-savart} leads to that 
\begin{align*}
\begin{split}
&v_n 
= \mathcal{V}_n[\omega_n] 
= c_n \mathcal{V}_n\big[r\mapsto r^{-\frac{\delta}2} K_{\xi_n}(\sqrt{\lambda} r)\big] \\
&
\text{and}
\qquad
d_n[\omega_n]
= c_n d_n\big[r\mapsto r^{-\frac{\delta}2} K_{\xi_n}(\sqrt{\lambda} r)\big]
= 0,
\end{split}
\end{align*}
with the notations in \eqref{def.psi_n}--\eqref{def.Biot-Savart}. The former condition implies that $c_n$ is nonzero since $v_n$ is assumed to be nontrivial. The latter one can be written equivalently to 
$$
c_n F_n(\sqrt{\lambda})=0. 
$$
Thus we have that $F_n(\sqrt{\lambda})=0$ since $c_n\neq0$. This completes the proof of the only if part.

For the if part, let $F_n(\sqrt{\lambda})=0$. Then, for any nonzero $c_n$, the vector field 
$$
v_n 
= c_n \mathcal{V}_n\big[r\mapsto r^{-\frac{\delta}2} K_{\xi_n}(\sqrt{\lambda} r)\big] 
$$ 
gives a nontrivial solution of $(\lambda + {\mathbb A}_{V,n}) v_n = 0$. Indeed, from the proof of the only if part, we ensure that $v_n$ is smooth and belongs to $D({\mathbb A}_{V,n})$. Note that the no-slip condition $v_n{|_{\partial\Omega}}=0$ is verified by the assumption that $F_n(\sqrt{\lambda})=0$. Moreover, setting 
$$
f_n=\lambda v_n - \Delta v_n + V^\bot \oprot v_n,
$$
from \eqref{properties.Biot-Savart}, we see that 
$$
\oprot f_n
=
\lambda (\oprot v_n) 
- \Delta (\oprot v_n)
+ V\cdot\nabla(\oprot v_n)
= 0. 
$$
Thus Proposition \ref{prop.rot.zero} yields that there is a function $p\in \widehat{W}^{1,2}(\Omega)$ such that $f_n = -\nabla p$. Operating the Helmholtz projection ${\mathbb P}$ to this equality, we find that $(\lambda + {\mathbb A}_{V,n}) v_n = 0$. This completes the proof of the if part. The proof of Proposition \ref{prop.charact.disc.ev} is complete.
\end{proof}
%

The following is a corollary of Propositions \ref{prop.spec.general} (\ref{item3.prop.spec.general}), \ref{prop.spec.small} (\ref{item3.prop.spec.small}) and \ref{prop.charact.disc.ev}. 
%
\begin{corollary}\label{cor.prop.charact.disc.ev}
Let $\alpha,\delta\in\R$ be sufficiently small. We have 
$$
\sigma_{{\rm disc}}(-{\mathbb A}_V) \cap \Sigma_{\frac{3}{4}\pi}
= \bigcup_{|n|=1} 
\Big\{\lambda\in\Sigma_{\frac{3}{4}\pi}~\Big|~F_n(\sqrt{\lambda})=0\Big\}. 
$$ 
\end{corollary}
%

\section{Quantitative analysis of discrete spectrum}\label{sec.quant.anal.disc.spec}

In this section, keeping Corollary \ref{cor.prop.charact.disc.ev} in mind, we analyze zeros of the analytic function $F_n(\sqrt{\lambda})$ with $|n|=1$ defined in \eqref{def.Fn}. Thanks to Proposition \ref{prop.spec.small}, it suffices to consider the zeros in disks centered at the origin with radius exponentially small in $|\alpha|$. The main result is Proposition \ref{prop.est.Fn}. The proof is based on asymptotic analysis under the smallness of $\alpha,\delta$.

Note that one can recover the results in \cite{Maekawa(2017a), Higaki(2019)} by putting $\delta=0$ in the statements of this section. However, this observation is not useful in the proof since we need to describe precisely the zeros of functions having multiple parameters. A continuity argument is not enough and quantitative analysis is needed. In fact, it is revealed that situations are different depending on the sign of $\delta$, and that the case $\delta<0$ seems to be more delicate.

\subsection{Expansion of the order}\label{subsec.order}

When $|n|=1$, we denote 
\begin{align}\label{def.1delta.eta}
1_\delta = \Big\{1 + \Big(\frac{\delta}{2}\Big)^2\Big\}^\frac12, 
\qquad 
\eta_n = \xi_n - 1. 
\end{align}
Here $\xi_n$ is defined in \eqref{def.xi}. A direct computation shows that 
\begin{align*}
\begin{split}
\Re(\xi_n) 
&= 
\frac{1_\delta}{\sqrt{2}} 
\bigg[
\Big\{1 + \Big(\frac{\alpha}{1_\delta^2}\Big)^2\Big\}^\frac12 + 1 
\bigg]^{\frac12}, \\
\Im(\xi_n) 
&= 
\opsgn(\alpha n) 
\frac{1_\delta}{\sqrt{2}} 
\bigg[
\Big\{1 + \Big(\frac{\alpha}{1_\delta^2}\Big)^2\Big\}^\frac12 - 1 
\bigg]^{\frac12}. 
\end{split}
\end{align*}

We need the following expansion of $\eta_n$ in the next subsection.

%
\begin{lemma}\label{lem.eta.asymptot.}
Let $|n|=1$. For sufficiently small $\alpha,\delta\in\R$, we have 
\begin{align}
\Re(\eta_n) 
&=
\frac{\alpha^2 + \delta^2}{8} 
+ O(\alpha^4 + \delta^4),
\label{est1.lem.eta.asymptot.} \\
\Im(\eta_n)
&=
\opsgn(\alpha n) \frac{|\alpha|}{2}
+ O\big( |\alpha| (\alpha^2 + \delta^2)\big). 
\label{est2.lem.eta.asymptot.}
\end{align}
All the implicit constants in $O(\cdot)$ are independent of $\alpha,\delta$. 
\end{lemma}
%
%
\begin{proof}
The proof is done by the Taylor theorem. For \eqref{est1.lem.eta.asymptot.}, from 
$$
\Re(\eta_n) 
= \Re(\xi_n) - 1
= 1_\delta
\bigg\{ 
1 
+ \frac{1}{8} \Big(\frac{\alpha}{1_\delta^2}\Big)^2 
- \frac{5}{128} \Big(\frac{\alpha}{1_\delta^2}\Big)^4
+ O(\alpha^6)
\bigg\}
- 1, 
$$
we see that 
$$
\Re(\eta_n) 
= 
1_\delta - 1 
+ \frac{\alpha^2}{8} 
+ \frac{\alpha^2}{8} \Big(\frac{1}{1_\delta^3} - 1\Big) 
- \frac{5\alpha^4}{128}  
- \frac{5\alpha^4}{128} \Big(\frac{1}{1_\delta^7} - 1\Big)
+ O(\alpha^6). 
$$
Hence \eqref{est1.lem.eta.asymptot.} is obtained by 
\begin{align*}
1_\delta - 1 
= \frac{\delta^2}8 - \frac{\delta^4}{128} + O(\delta^6) 
\end{align*}
and
\begin{align*}
\frac{1}{1_\delta^3} - 1 
= -\frac{3\delta^2}{8} + O(\delta^4), 
\qquad 
\frac{1}{1_\delta^7} - 1 
= -\frac{7\delta^2}{8} + O(\delta^4). 
\end{align*}
For \eqref{est2.lem.eta.asymptot.}, from 
$$
\Im(\eta_n) 
= \Im(\xi_n) 
= 
\opsgn(\alpha n) 
1_\delta
\bigg\{ 
\frac{1}{2} \Big|\frac{\alpha}{1_\delta^2}\Big| 
- \frac{1}{16} \Big|\frac{\alpha}{1_\delta^2}\Big|^3
+ O(|\alpha|^5)
\bigg\}, 
$$
we see that 
$$
\Im(\eta_n) 
= 
\opsgn(\alpha n) 
\bigg\{ 
\frac{|\alpha|}{2} 
+ \frac{|\alpha|}{2} \Big(\frac{1}{1_\delta} - 1\Big) 
- \frac{|\alpha|^3}{16} 
- \frac{|\alpha|^3}{16} \Big(\frac{1}{1_\delta^5} - 1\Big) 
+ O(|\alpha|^5) 
\bigg\}. 
$$
Hence \eqref{est2.lem.eta.asymptot.} is obtained by 
\begin{align*}
\frac{1}{1_\delta} - 1 
= -\frac{\delta^2}{8} + O(\delta^4), 
\qquad 
\frac{1}{1_\delta^5} - 1 
= -\frac{5\delta^2}{8} + O(\delta^4). 
\end{align*}
This completes the proof. 
\end{proof}
%

\subsection{Asymptotic analysis}\label{subsec.asymptot.anal.}

We consider $F_n(\sqrt{\lambda})$ in \eqref{def.Fn} with $|n|=1$, namely, the function 
\begin{align}\label{def.Fn.z}
F_n(z) = \int_{1}^{\infty} s^{-\frac{\delta}{2}} K_{1+\eta_n}(zs) \dd s, 
\quad 
z\in\Sigma_{\frac{\pi}2}. 
\end{align}
%

%
\begin{lemma}\label{lem.Fn.1}
Let $|n|=1$. For $\alpha,\delta\in\R$, we have 
\begin{align}\label{eq1.lem.Fn.1}
\Big(\frac{\delta}{2}+\eta_n\Big) F_n(z) 
= 
K_{1+\eta_n}(z) 
- z\int_{1}^{\infty} s^{1-\frac{\delta}{2}} K_{\eta_n}(zs) \dd s,
\quad 
z\in\Sigma_{\frac{\pi}2}. 
\end{align}
\end{lemma}
%
%
\begin{proof}
By the recurrence relation (see \cite[Chapter $\mathrm{I}\hspace{-1.2pt}\mathrm{I}\hspace{-1.2pt}\mathrm{I}$ 3$\cdot$71 (3)]{Watson(1944)})
$$
\mu K_{\mu}(z)
= - z \frac{\dd K_{\mu}}{\dd z}(z) 
- z K_{\mu-1}(z), 
$$
we have  
$$
(1+\eta_n) K_{1+\eta_n}(zs) 
= -s \frac{\dd}{\dd s} K_{1+\eta_n}(zs) 
- zs K_{\eta_n}(zs). 
$$
Thus the definition \eqref{def.Fn.z} and integration by parts give 
\begin{align*}
\begin{split}
(1+\eta_n) F_n(z) 
&= 
\int_{1}^{\infty} s^{-\frac{\delta}{2}} 
(1+\eta_n) K_{1+\eta_n}(zs) \dd s \\
&= 
\int_{1}^{\infty} s^{-\frac{\delta}{2}} 
\Big(-s \frac{\dd}{\dd s} K_{1+\eta_n}(zs) 
- zs K_{\eta_n}(zs)\Big) \dd s \\
&= 
K_{1+\eta_n}(z) 
+ \Big(1 - \frac{\delta}{2}\Big) F_n(z) 
- z\int_{1}^{\infty} s^{1-\frac{\delta}{2}} K_{\eta_n}(zs) \dd s, 
\end{split}
\end{align*}
which implies the assertion of the lemma. 
\end{proof}
%

Using the relation \eqref{eq1.lem.Fn.1}, we investigate zeros of $F_n(z)$ near the origin. We perform asymptotic analysis when $|z|$ is sufficiently small. Since the asymptotics of $K_{1+\eta_n}(z)$ is already obtained in Lemma \ref{lem.est.bessel2} (\ref{item1.lem.est.bessel2}), we focus on the second term on the right-hand side of \eqref{eq1.lem.Fn.1}. In what follows in this section, we assume smallness of $\alpha,\delta$. Although some estimates can be proved under weaker assumptions, we will not give the details for simplicity.
%
\begin{lemma}\label{lem.Fn.2}
Let $|n|=1$. For sufficiently small $\alpha,\delta\in\R$, we have 
\begin{align}\label{eq1.lem.Fn.2}
z\int_{1}^{\infty} s^{1-\frac{\delta}{2}} K_{\eta_n}(zs) \dd s 
&= 
\frac{\Delta(\delta,\eta_n)}{2} \Big(\frac{z}2\Big)^{-1+\frac{\delta}2} 
+ R^{(2)}_n(z), 
\quad
z \in\Sigma_{\frac{\pi}2} \cap \{|z|<1\}.
\end{align}
Here $\Delta(\delta,\eta_n)$ is defined by 
\begin{align}\label{def.Delta}
\Delta(\delta,\eta_n) 
= \Gamma\Big(1-\frac{\delta}4-\frac{\eta_n}2\Big) 
\Gamma\Big(1-\frac{\delta}4+\frac{\eta_n}2\Big), 
\end{align}
where $\Gamma(z)$ is the Gamma function, and $R^{(2)}_n(z)$ is the remainder and satisfies 
\begin{align}\label{est1.lem.Fn.2}
|R^{(2)}_n(z)| 
\le 
C|z|^{1-\Re \eta_n} 
\big(1 + \big|\log|z|\big|\big), 
\quad 
z \in\Sigma_{\frac{\pi}2} \cap \{|z|<1\}.
\end{align}
The constant $C$ is independent of $\alpha,\delta$. 
\end{lemma}
%
%
\begin{proof}
If we show that 
\begin{align}\label{eq1.proof.lem.Fn.2}
\begin{split}
z\int_{0}^{\infty} s^{1-\frac{\delta}{2}} K_{\eta_n}(zs) \dd s 
= \frac{\Delta(\delta,\eta_n)}{2} \Big(\frac{z}2\Big)^{-1+\frac{\delta}2}, 
\end{split}
\end{align}
the assertion follows. Indeed, it is not hard to check that 
$$
R^{(2)}_n(z)
:=
-z \int_{0}^{1} s^{1-\frac{\delta}{2}} K_{\eta_n}(zs) \dd s 
$$
satisfies \eqref{est1.lem.Fn.2} using the estimates in Lemma \ref{lem.est.bessel2} (\ref{item2.lem.est.bessel2}).

By the representation \eqref{rep.K.int} and by the Fubini theorem, we have 
\begin{align*}
\begin{split}
&z\int_{0}^{\infty} s^{1-\frac{\delta}{2}} K_{\eta_n}(zs) \dd s \\
&=
z\int_{0}^{\infty} s^{1-\frac{\delta}{2}} 
\bigg(\frac12 \int_{0}^{\infty} e^{-\frac{zs}2(t+\frac1{t})} t^{-\eta_n-1} \dd t \bigg) \dd s \\
&=
\frac{z}2 
\int_{0}^{\infty} 
t^{-1-\eta_n} 
\bigg(
\int_{0}^{\infty} s^{1-\frac{\delta}{2}} e^{-\frac{z}2(t+\frac1{t})s} \dd s\bigg) \dd t. 
\end{split}
\end{align*}
Observing that 
\begin{align*}
\int_{0}^{\infty} 
s^{1-\frac{\delta}{2}} e^{-as} \dd s
= \Gamma\Big(2-\frac{\delta}2\Big) a^{-2+\frac{\delta}2},
\quad 
a\in\Sigma_{\frac{\pi}2}, 
\end{align*}
we have 
\begin{align*}
\begin{split}
&
\frac{z}2 
\int_{0}^{\infty} 
t^{-1-\eta_n} 
\bigg(
\int_{0}^{\infty} 
s^{1-\frac{\delta}{2}} e^{-\frac{z}2(t+\frac1{t})s} \dd s\bigg) \dd t \\
&= 
\Gamma\Big(2-\frac{\delta}2\Big) 
\Big(\frac{z}2\Big)^{-1+\frac{\delta}2} 
\int_{0}^{\infty}
t^{-1-\eta_n} 
\Big(t+\frac1{t}\Big)^{-2+\frac{\delta}2}
\dd t \\
&= 
\Gamma\Big(2-\frac{\delta}2\Big) 
\Big(\frac{z}2\Big)^{-1+\frac{\delta}2} 
\int_{0}^{\infty}
\frac{t^{1-\frac{\delta}2-\eta_n}}{(t^2+1)^{2-\frac{\delta}2}} 
\dd t. 
\end{split}
\end{align*}
The change of variable $t=\tau^\frac12$ leads to 
\begin{align*}
\begin{split}
\int_{0}^{\infty}
\frac{t^{1-\frac{\delta}2-\eta_n}}{(t^2+1)^{2-\frac{\delta}2}} 
\dd t
&= \frac12 
\int_{0}^{\infty}
\frac{\tau^{-\frac{\delta}4-\frac{\eta_n}2}}{(\tau+1)^{2-\frac{\delta}{2}}} 
\dd \tau \\
&= 
\frac12 
B\Big(1-\frac{\delta}4-\frac{\eta_n}2, 
1-\frac{\delta}4+\frac{\eta_n}2\Big), 
\end{split}
\end{align*}
where $B(p,q)$ is the Beta function. Then the well-known formulas
\begin{align*}
z\Gamma(z) = \Gamma(z+1), \qquad 
B(p,q) = \frac{\Gamma(p)\Gamma(q)}{\Gamma(p+q)} 
\end{align*}
imply \eqref{eq1.proof.lem.Fn.2}. This completes the proof. 
\end{proof}
%

%
\begin{corollary}\label{cor.lem.Fn.2}
Let $|n|=1$. For sufficiently small $\alpha,\delta\in\R$, we have 
\begin{align}\label{eq1.cor.lem.Fn.2}
\begin{split}
\Big(\frac{\delta}{2}+\eta_n\Big) F_n(z) 
&= 
\frac{\Gamma(1+\eta_n)}{2} \Big(\frac{z}{2}\Big)^{-1-\eta_n} 
- \frac{\Delta(\delta,\eta_n)}{2} \Big(\frac{z}2\Big)^{-1+\frac{\delta}2} \\
&\quad
+ R^{(3)}_n(z),
\quad
z \in\Sigma_{\frac{\pi}2} \cap \{|z|<1\}. 
\end{split}
\end{align}
Here $R^{(3)}_n$ is the remainder and satisfies
\begin{align}\label{est1.cor.lem.Fn.2}
|R^{(3)}_n(z)| 
\le 
C|z|^{1-\Re \eta_n} 
\big(1 + \big|\log|z|\big|\big), 
\quad 
z \in\Sigma_{\frac{\pi}2} \cap \{|z|<1\}. 
\end{align}
The constant $C$ is independent of $\alpha,\delta$. 
\end{corollary}
%
%
\begin{proof}
This is a consequence of the previous proposition and Lemma \ref{lem.est.bessel2} (\ref{item1.lem.est.bessel2}). 
\end{proof}
%

Proposition \ref{prop.est.Fn} below, giving a lower bound of $|F_n(z)|$, is proved based on the expansion \eqref{eq1.cor.lem.Fn.2}. In the proof, we need precise estimates of the coefficients appearing in \eqref{eq1.cor.lem.Fn.2}. 
%
\begin{lemma}\label{lem.Gamma.Delta.asymptot.}
Let $|n|=1$. For sufficiently small $\alpha,\delta\in\R$, we have 
\begin{align}
\opLog \Gamma(1+\eta_n) 
&= -\gamma \eta_n + O(|\eta_n|^2), 
\label{est1.lem.Gamma.Delta.asymptot.} \\
\opLog \Delta(\delta,\eta_n) 
&= \gamma \Big(\frac{\delta}2\Big) 
+  O\Big(\Big(\frac{\delta}2\Big)^2+|\eta_n|^2\Big)
\label{est2.lem.Gamma.Delta.asymptot.}, 
\end{align}
where $\gamma = 0.5772\ldots$ is the Euler constant. Moreover, if $\delta\ge0$, 
\begin{align}\label{est3.lem.Gamma.Delta.asymptot.}
\opLog \Delta(\delta,\eta_n) 
- \opLog \Gamma(1+\eta_n)
&= \gamma \Big(\frac{\delta}2+\eta_n\Big) 
+ O\Big(\Big|\frac{\delta}2+\eta_n\Big|^2\Big). 
\end{align}
All the implicit constants in $O(\cdot)$ are independent of $\alpha,\delta$. 
\end{lemma}
%
%
\begin{proof}
We may apply the Taylor series expansion of $\opLog\Gamma(1+z)$ 
\begin{align}\label{eq1.proof.lem.Gamma.Delta.asymptot.}
\opLog\Gamma(1+z) 
= \gamma (-z) 
+ \sum_{k=2}^\infty \frac{\zeta(k)}{k} (-z)^k, 
\quad 
\{|z|<1\}. 
\end{align}
Here $\zeta(k)=\sum_{m=1}^\infty m^{-k}$ is the Riemann zeta function. One can prove \eqref{eq1.proof.lem.Gamma.Delta.asymptot.} using 
\begin{align*}
z\Gamma(z) = \Gamma(z+1), \qquad 
\frac{1}{\Gamma(z)} 
= ze^{\gamma z} \prod_{m=1}^\infty \Big(1+\frac{z}{m}\Big) e^{-\frac{z}{m}}. 
\end{align*}
Indeed, from 
\begin{align*}
\opLog\Gamma(1+z)
&= \opLog z - \opLog \frac{1}{\Gamma(z)} \\
&= 
\gamma (-z)
- \sum_{m=1}^\infty
\bigg(
\opLog \Big(1+\frac{z}{m}\Big) 
- \frac{z}{m} 
\bigg)
\end{align*}
and the Taylor series expansion
\begin{align*}
\opLog (1+z) 
= -\sum_{k=1}^\infty 
\frac{1}{k} 
(-z)^{k}, 
\quad 
\{|z|<1\}, 
\end{align*}
we see that 
\begin{align*}
\opLog\Gamma(1+z)
= 
\gamma (-z) 
+ \sum_{m=1}^\infty \sum_{k=2}^\infty 
\frac{1}{k} 
\Big(-\frac{z}{m}\Big)^k, 
\end{align*}
which leads to \eqref{eq1.proof.lem.Gamma.Delta.asymptot.} after change of order of summations.

The expansion \eqref{est1.lem.Gamma.Delta.asymptot.} is a direct consequence of \eqref{eq1.proof.lem.Gamma.Delta.asymptot.}. Also, by 
\begin{align*}
\opLog \Gamma\Big(1-\frac{\delta}4-\frac{\eta_n}2\Big) 
&= 
\frac{\gamma}{2} \Big(\frac{\delta}2 + \eta_n\Big)
+ O\Big(\Big|\frac{\delta}2 + \eta_n\Big|^2\Big), \\
\opLog \Gamma\Big(1-\frac{\delta}4+\frac{\eta_n}2\Big) 
&= \frac{\gamma}{2} \Big(\frac{\delta}2 - \eta_n\Big)
+ O\Big(\Big|\frac{\delta}2 - \eta_n\Big|^2\Big) 
\end{align*}
and the definition of $\Delta(\delta,\eta_n)$ in \eqref{def.Delta}, we have 
\begin{align*}
\opLog \Delta(\delta,\eta_n)
&= \opLog \Gamma\Big(1-\frac{\delta}4-\frac{\eta_n}2\Big) 
+ \opLog \Gamma\Big(1-\frac{\delta}4+\frac{\eta_n}2\Big) \\
&= \gamma \Big(\frac{\delta}2\Big) 
+  O\Big(\Big|\frac{\delta}2 + \eta_n\Big|^2 + \Big|\frac{\delta}2 - \eta_n\Big|^2\Big), 
\end{align*}
which implies \eqref{est2.lem.Gamma.Delta.asymptot.}. If $\delta\ge0$, we see from Lemma \ref{lem.eta.asymptot.} that, for sufficiently small $\alpha,\delta$, 
\begin{align*}
\Big(\frac{\delta}2\Big)^2 + |\eta_n|^2
= 
\Big|\frac{\delta}2+\eta_n\Big|^2 
- \delta \Re \eta_n
\le 
\Big|\frac{\delta}2+\eta_n\Big|^2, 
\end{align*}
which implies \eqref{est3.lem.Gamma.Delta.asymptot.}. This completes the proof. 
\end{proof}
%

The following is the key technical lemma in the proof of Proposition \ref{prop.est.Fn} below.

%
\begin{lemma}\label{lem.lowerbd}
Let $\ep\in(0,\frac{\pi}2)$. Suppose that $\zeta\in\C$ with $|\zeta|\ll1$ satisfies 
\begin{align}\label{cond1.lem.lowerbd}
\Re\zeta>0, \qquad |\Im\zeta|>0 
\end{align}
and 
\begin{align}\label{cond2.lem.lowerbd}
\Big\{
\Re\zeta 
+ (1+\kappa) \frac{(\Im\zeta)^2}{\Re\zeta}
\Big\}
\Big( \frac{\pi}{2} - \ep \Big)
< \pi 
\end{align}
with some constant $\kappa=\kappa(\ep)\in(0,\frac12)$ independent of $\zeta$. Then, by defining 
\begin{align}\label{def.K.lem.lowerbd}
K(\zeta) 
= 
\min\bigg\{
\Big\{
\frac{(\Re\zeta)^2}{|\Im\zeta|}
+ |\Im\zeta| 
\Big\}, 
\Re\zeta
\bigg\}, 
\end{align}
one has 
\begin{align}\label{est.lem.lowerbd}
|1 - w^{\zeta}| 
\ge
C \min
\Big\{
1, 
K(\zeta)
\big|\log|w|\big|
\Big\}, 
\quad 
w\in\Sigma_{\frac{\pi}2-\ep} \cap \{|z|<1\}.
\end{align}
The constant $C$ depends only on $\ep$ and $\kappa$. 
\end{lemma}
%
%
\begin{proof}
By setting 
\begin{align}
\mu&= (\Re\zeta) \log|w| - (\Im\zeta) \oparg w, 
\label{def.mu.proof.lem.lowerbd} \\
\theta&= (\Im\zeta) \log|w| + (\Re\zeta) \oparg w, 
\label{def.theta.proof.lem.lowerbd}
\end{align}
we denote 
$$
1 - w^{\zeta} = 1 - e^{\mu} e^{i\theta}.
$$
From 
$$
\log |w| 
= \frac{\Im\zeta}{\Re\zeta} \oparg w 
+ \frac{1}{\Re\zeta} \mu, 
$$
we compute 
\begin{align}\label{theta.proof.lem.lowerbd}
\theta 
= \Big\{
\Re\zeta 
+ \frac{(\Im\zeta)^2}{\Re\zeta}
\Big\} \oparg w 
+ \frac{\Im\zeta}{\Re\zeta} \mu. 
\end{align}
Before going into details, let us explain the difficulties. When $\mu$ is close to zero, one essentially needs to provide lower bounds of  $|1-e^{i\theta}|$. However, such bounds require good control of $\theta$, since $1-e^{i\theta}$ vanishes when $\theta=2m\pi$ with $m\in\Z$. The reason why the conditions \eqref{cond1.lem.lowerbd}--\eqref{cond2.lem.lowerbd} are needed is to control the range of $\theta$ when $\mu$ is close to zero.

We will consider two cases:

(i) Case $|\mu| \le \kappa |\Im \zeta| |\oparg w|$. In this case, we have 
$$
\frac12 \le e^{\mu} \le \frac32.
$$
In addition, by \eqref{theta.proof.lem.lowerbd} and the assumption \eqref{cond2.lem.lowerbd}, 
$$
|\theta|
\le  
\Big\{
\Re\zeta 
+ (1+\kappa) \frac{(\Im\zeta)^2}{\Re\zeta}
\Big\}
|\oparg w| 
<\pi,
\quad 
w\in\Sigma_{\frac{\pi}2-\ep} \cap \{|z|<1\}.
$$
Thus $e^{i\theta}$ is equal to $1$ if and only if $\theta=0$. If $0\le |\theta|< \frac{\pi}2$, the imaginary part gives 
\begin{align}\label{lowerbd.im.proof.lem.lowerbd}
|1 - e^{\mu} e^{i\theta}|
\ge
e^{\mu} |\sin\theta| 
\ge 
e^{\mu} \frac{2}{\pi} |\theta|
\ge 
\frac{1}{\pi} |\theta|. 
\end{align}
If $\frac{\pi}2 \le |\theta|< \pi$, the real part gives 
\begin{align}\label{lowerbd.re.proof.lem.lowerbd}
|1 - e^{\mu} e^{i\theta}| 
\ge
|1 - e^{\mu} \cos\theta| 
\ge 
1 
>
\frac{1}{\pi} |\theta|. 
\end{align}
Hence we estimate $|\theta|$. Combining $|\mu| \le \kappa |\Im \zeta| |\oparg w|$ with \eqref{theta.proof.lem.lowerbd}, we have 
$$
|\theta|
\ge 
\Big\{
\Re\zeta + (1-\kappa) \frac{(\Im\zeta)^2}{\Re\zeta}
\Big\} 
|\oparg w|. 
$$
Combining with \eqref{def.mu.proof.lem.lowerbd}, 
\begin{align*}
\Re\zeta \big|\log|w|\big|
\le 
|\Im\zeta| |\oparg w| + |\mu|
\le 
(1+\kappa) |\Im\zeta| |\oparg w|. 
\end{align*}
By these two estimates, we obtain 
\begin{align*}
\begin{split}
|\theta|
&\ge 
\Big\{
\Re\zeta + (1-\kappa) \frac{(\Im\zeta)^2}{\Re\zeta}
\Big\} 
\frac{\Re\zeta \big|\log|w|\big|}{(1+\kappa) |\Im\zeta|} \\
&\ge 
\frac{1-\kappa}{1+\kappa} 
\Big\{
\frac{(\Re\zeta)^2}{|\Im\zeta|}
+ |\Im\zeta| 
\Big\} 
\big|\log|w|\big|. 
\end{split}
\end{align*}
Therefore, from \eqref{lowerbd.im.proof.lem.lowerbd} and \eqref{lowerbd.re.proof.lem.lowerbd}, we see that 
\begin{align}\label{lowerbd1.proof.lem.lowerbd}
|1 - e^{\mu} e^{i\theta}|
\ge
\frac{1}{\pi} 
\frac{1-\kappa}{1+\kappa} 
\Big\{
\frac{(\Re\zeta)^2}{|\Im\zeta|}
+ |\Im\zeta| 
\Big\} 
\big|\log|w|\big|. 
\end{align}

(ii) Case $|\mu| > \kappa |\Im \zeta| |\oparg w|$. In this case, we may rely on 
\begin{align}\label{lowerbd.e.proof.lem.lowerbd}
|1 - e^{\mu} e^{i\theta}|
\ge
|1 - e^{\mu}|
\ge
e^{-1} \min\{1, |\mu|\}, 
\quad
\mu,\theta\in\R. 
\end{align}
We deduce that if $|\oparg w| > \frac12 \frac{\Re\zeta}{|\Im\zeta|} \big|\log |w|\big|$, 
$$
|\mu| > \frac{\kappa}2 \Re\zeta \big|\log |w|\big| 
$$
by $|\mu| > \kappa |\Im \zeta| |\oparg w|$, and that if $|\oparg w| \le \frac12 \frac{\Re\zeta}{|\Im\zeta|} \big|\log |w|\big|$, 
$$
|\mu| 
\ge 
\Re\zeta \big|\log|w|\big| 
- |\Im\zeta| |\oparg w|
\ge 
\frac12 \Re\zeta \big|\log|w|\big|. 
$$
by \eqref{def.mu.proof.lem.lowerbd}. Combining these two with \eqref{lowerbd.e.proof.lem.lowerbd}, we obtain 
\begin{align}\label{lowerbd2.proof.lem.lowerbd}
|1 - e^{\mu} e^{i\theta}|
\ge 
e^{-1} \min
\Big\{
1, \frac{\kappa}2 \Re\zeta \big|\log |w|\big| \Big\}. 
\end{align}
The assertion follows from \eqref{lowerbd1.proof.lem.lowerbd} and \eqref{lowerbd2.proof.lem.lowerbd}. The proof is complete. 
\end{proof}
%

%
\begin{proposition}\label{prop.est.Fn}
Let $|n|=1$ and $\ep\in(0,\frac{\pi}2)$. Let $K(\zeta)$ be defined in \eqref{def.K.lem.lowerbd}. 
For sufficiently small $(\alpha,\delta)\in\R^\ast\times\R_{\ge0}$, we have 
\begin{align}\label{est1.prop.est.Fn}
\begin{split}
\Big|\Big(\frac{\delta}{2}+\eta_n\Big) F_n(z)\Big| 
\ge
C |z|^{-1-\Re\eta_n} 
\min
\Big\{
1, K\Big(\frac{\delta}{2}+\eta_n\Big) \big|\log|z|\big|
\Big\}, & \\
z\in \Sigma_{\frac{\pi}2-\ep}
\cap \Big\{|z|<K\Big(\frac{\delta}{2}+\eta_n\Big)\Big\}.& 
\end{split}
\end{align}
The constant $C$ depends only on $\ep$. 
\end{proposition}
%
%
\begin{remark}\label{rem.prop.est.Fn}
One observes a sort of stabilizing effect by the flow $\delta W$ from this proposition. By the definition \eqref{def.K.lem.lowerbd} and Lemma \ref{lem.eta.asymptot.}, we have a simple (but rough) estimate from below 
\begin{align}\label{est1.rem.prop.est.Fn}
\begin{split}
K\Big(\frac{\delta}{2}+\eta_n\Big) 
&\ge 
\min
\Big\{\Big|\Im\Big(\frac{\delta}{2}+\eta_n\Big)\Big|, 
\Re\Big(\frac{\delta}{2}+\eta_n\Big)
\Big\} \\
&\ge 
\frac18 \min\{|\alpha|,\delta+\alpha^2\}. 
\end{split}
\end{align}
The second inequality is valid for sufficiently small $\alpha,\delta$. Therefore, the radius of the disks on which $F_n(z)$ has no zeros is greater than that for $\delta=0$. This is interpreted as a stabilizing effect by $\delta W$ in time frequency near zero related to large-time behavior of flows.
\end{remark}
%
%
\begin{proof}
Let $z\in \Sigma_{\frac{\pi}2-\ep} \cap \{|z|<\frac12\}$ first. Using Corollary \ref{cor.lem.Fn.2}, we write
\begin{align}\label{eq1.proof.prop.est.Fn}
\Big(\frac{\delta}{2}+\eta_n\Big)
F_n(z)
&= 
\frac{\Gamma(1+\eta_n)}{2} \Big(\frac{z}{2}\Big)^{-1-\eta_n} 
\Big\{
1 
- 
\frac{\Delta(\delta,\eta_n)}{\Gamma(1+\eta_n)} \Big(\frac{z}2\Big)^{\frac{\delta}2+\eta_n} 
+ R_n(z) 
\Big\}. 
\end{align}
Here  
$$
R_n(z) 
= \frac{2}{\Gamma(1+\eta_n)}
\Big(\frac{z}{2}\Big)^{1+\eta_n} R^{(3)}_n(z) 
$$
is the remainder and satisfies 
\begin{align}\label{est1.proof.prop.est.Fn}
|R_n(z)| 
\le 
C|z|^{2} \big|\log|z|\big|, 
\quad 
z \in\Sigma_{\frac{\pi}2} 
\cap 
\Big\{|z|<\frac12\Big\}. 
\end{align}
The condition $\delta\ge0$ and Lemma \ref{lem.Gamma.Delta.asymptot.} imply 
$$
\frac{\Delta(\delta,\eta_n)}{\Gamma(1+\eta_n)} 
= e^{\opLog \Delta(\delta,\eta_n) 
- \opLog \Gamma(1+\eta_n)}
= e^{\gamma_n (\frac{\delta}2+\eta_n)}, 
$$
where $\gamma_n=\gamma_n(\delta,\eta_n)$ is the function satisfying 
$$
\gamma_n 
= 
\gamma 
+ O\Big(\Big|\frac{\delta}2+\eta_n\Big|\Big). 
$$

Setting 
$$
\zeta_n = \frac{\delta}2+\eta_n, 
\qquad w_n = \frac{e^{\gamma_n}}2 z, 
$$
we will derive a lower bound of 
$$
1 
- 
\frac{\Delta(\delta,\eta_n)}{\Gamma(1+\eta_n)} \Big(\frac{z}2\Big)^{\frac{\delta}2+\eta_n} 
= 1 - w_n^{\zeta_n}. 
$$
To apply Lemma \ref{lem.lowerbd}, we check that all the conditions are fulfilled by $\zeta_n, w_n$. We have 
$$
|w_n| 
\le 1, 
\qquad
|\oparg w_n| 
\le \frac{\pi}2-\frac{\ep}2 
$$
for sufficiently small $\alpha,\delta$. We also have \eqref{cond1.lem.lowerbd} by Lemma \ref{lem.eta.asymptot.}. By the same lemma, there are constants $C_1,C_2$ independent of $\alpha,\delta$ such that 
$$
\frac{(\Im\zeta_n)^2}{\Re\zeta_n} 
\le
\Big\{
\frac{|\alpha|}{2}
+ C_1 |\alpha| (\alpha^2 + \delta^2)\Big\}^2 
\Big\{\frac{\alpha^2 + \delta^2}{8} - C_2 (\alpha^4+\delta^4)\Big\}^{-1}. 
$$
Thus, for sufficiently small $\alpha,\delta$, we have 
$$
\frac{(\Im\zeta_n)^2}{\Re\zeta_n} 
= 2 + O(\alpha^2 + \delta^2) 
$$
and, with a constant $\kappa_n=\kappa_n(\ep)\in(0,\frac12)$ independent of $\alpha,\delta$, 
$$
\Big\{
\Re\zeta_n 
+ (1+\kappa_n) \frac{(\Im\zeta_n)^2}{\Re\zeta_n}
\Big\}
\Big(\frac{\pi}2-\frac{\ep}2\Big)
< \pi, 
$$
which is \eqref{cond2.lem.lowerbd} with $\ep$ replaced by $\frac\ep2$. Now, applying Lemma \ref{lem.lowerbd}, we see that  
\begin{align*}
|1 - w_n^{\zeta_n}| 
&\ge
C \min
\Big\{
1, 
K(\zeta_n)
\big|\log|w_n^{\zeta_n}|\big|
\Big\} \\
&\ge
C \min
\Big\{
1, 
K(\zeta_n)
\big|\log|z|\big|
\Big\},
\quad 
z\in \Sigma_{\frac{\pi}2-\ep} 
\cap 
\Big\{|z|<\frac12\Big\}, 
\end{align*}
for sufficiently small $\alpha,\delta$. The constant $C$ depends only on $\ep$.

Therefore, combining this estimate with \eqref{eq1.proof.prop.est.Fn} and \eqref{est1.proof.prop.est.Fn}, we obtain a lower bound 
\begin{align}\label{est2.proof.prop.est.Fn}
\begin{split}
|\zeta_n F_n(z)| 
\ge
C |z|^{-1-\Re\eta_n} 
\Big(
\min
\Big\{
1, 
K(\zeta_n)
\big|\log|z|\big|
\Big\} 
- |z|^{2} \big|\log|z|\big|
\Big), 
\end{split}
\end{align}
which implies the desired estimate \eqref{est1.prop.est.Fn}. The proof is complete. 
\end{proof}

We state two corollaries to this proposition. The first one gives a simpler version of \eqref{est1.prop.est.Fn} useful for later calculation. The second one uses the results in Section \ref{sec.spec.anal.}. 
%
\begin{corollary}\label{cor1.prop.est.Fn}
Let $|n|=1$ and $\ep\in(0,\pi)$. 
For sufficiently small $(\alpha,\delta)\in\R^\ast\times\R_{\ge0}$, we have 
\begin{align*}
\Big|
\Big(\frac{\delta}{2}+\eta_n\Big)
F_n(\sqrt{\lambda})
\Big| 
\ge
C |\lambda|^{-\frac{\Re\xi_n}{2}}
\min
\big\{
1, \alpha^2 \big|\log|\lambda|\big|
\big\},& \\
\quad 
\lambda\in \Sigma_{\pi-\ep}\cap \{|z|<\alpha^4\}.&
\end{align*}
In particular, 
$$
\frac{1}{|F_n(\sqrt{\lambda})|}
\le 
C 
\Big|
\frac1\alpha\Big(\frac{\delta}{2}+\eta_n\Big)
\Big|^{-1} 
|\lambda|^{\frac{\Re\xi_n}{2}}, 
\quad \lambda\in \Sigma_{\pi-\ep}
\cap 
\Big\{|z|<e^{-\frac{1}{4|\alpha|}}\Big\}. 
$$
The constant depends only on $\ep$. 
\end{corollary}
%
%
\begin{proof}
The assertion follows from \eqref{est2.proof.prop.est.Fn} combined with the simple lower bound \eqref{est1.rem.prop.est.Fn}. 
\end{proof}
%

%
\begin{corollary}\label{cor2.prop.est.Fn}
Let $\ep\in(0,\frac{\pi}4)$. For sufficiently small $(\alpha,\delta)\in\R^\ast\times\R_{\ge0}$, we have 
$$
\Sigma_{\frac34 \pi-\ep}
\subset 
\rho(-\mathbb{A}_{V}). 
$$ 
\end{corollary}
%
%
\begin{proof}
In view of Proposition \ref{prop.spec.small} and Corollary \ref{cor1.prop.est.Fn}, we set 
\begin{align*}
\begin{split}
\mathcal{S}_1(\alpha)
&=
\Big(
\Sigma_{\frac34\pi} + 4e\alpha^2 e^{-\frac{1}{4|\alpha|}}
\Big)
\cap 
\Big\{|z| > 8e\alpha^2 e^{-\frac{1}{4|\alpha|}} \Big\}, \\
\mathcal{S}_2(\alpha)
&= \Sigma_{\frac34\pi} 
\cap
\Big\{|z|<e^{-\frac{1}{4|\alpha|}}\Big\}. 
\end{split}
\end{align*}
From Propositions \ref{prop.spec.general} (\ref{item1.prop.spec.general}) and \ref{prop.spec.small}, and Corollary \ref{cor.prop.charact.disc.ev}, we see that both  $\mathcal{S}_1(\alpha)$ and $\mathcal{S}_2(\alpha)$ are contained in $\rho(-\mathbb{A}_{V})$ for sufficiently small $\alpha,\delta$. For a given $\ep\in(0,\frac{\pi}4)$, by an easy geometric consideration, we find that $\Sigma_{\frac34 \pi-\ep}$ is contained in $\mathcal{S}_1(\alpha)\cup\mathcal{S}_1(\alpha)$ if $\alpha$ is small enough depending on $\ep$. This implies the assertion. 
\end{proof}
%

\section{Resolvent estimate}\label{sec.resol.est}

In this section, we estimate the solutions of 
\begin{equation}\tag{R}\label{eq.resol}
(\lambda + \mathbb{A}_{V}) v = f 
\end{equation}
for given $\lambda\in\rho(-{\mathbb A}_{V})$ and $f\in L^2_{\sigma}(\Omega)$. The main result is the following. 
%
\begin{proposition}\label{prop.resol.LpLq}
Let $\ep\in(0,\frac{\pi}4)$ and let $(\alpha,\delta)\in\R^\ast\times\R_{\ge0}$ be sufficiently small. We have, for $q\in(1.2]$ and $f\in L^2_\sigma(\Omega) \cap L^q(\Omega)^2$, 
\begin{equation}\label{est1.prop.resol.LpLq}
\begin{split}
\|(\lambda+{\mathbb A}_{V})^{-1} f\|_{L^2} 
\le 
C |\lambda|^{-\frac32+\frac1q}
\|f\|_{L^q},
\quad 
\lambda\in \Sigma_{\frac34\pi-\ep} 
\end{split}
\end{equation}
and, for $f\in L^2_\sigma(\Omega)$, 
\begin{equation}\label{est2.prop.resol.LpLq}
\begin{split}
\|\nabla (\lambda+{\mathbb A}_{V})^{-1} f\|_{L^2} 
\le 
C |\lambda|^{-\frac12} 
\|f\|_{L^2},
\quad 
\lambda\in \Sigma_{\frac34\pi-\ep}. 
\end{split}
\end{equation}
The constant $C$ depends only on $\alpha,\delta,\ep,q$. 
\end{proposition}
%
Once Proposition \ref{prop.resol.LpLq} is proved, it is routine to prove Theorem \ref{thm.L2L2} by representing $\{e^{-t {\mathbb A}_V}\}_{t\ge0}$ in the Dunford integral of the resolvent. Thus the detail will be given in Appendix \ref{app.proof.thm.L2L2}.

We prove Proposition \ref{prop.resol.LpLq} in Subsection \ref{subsec.proof.prop.resol.LpLq} by a combination of energy method and explicit formulas for the solution. Note that the estimate \eqref{est1.prop.resol.LpLq} cannot be obtained by energy method alone, due to the absence of the Hardy inequality. However, this is not the case when $\lambda$ belongs to sectors shifted exponentially small in $|\alpha|$; see Proposition \ref{prop.resol.est.1} for details. Therefore, all that remains is to prove the estimate when $\lambda$ belongs to the intersection of sectors and the disks centered at the origin whose radius is exponentially small in $|\alpha|$. This proof is done by explicit formulas; see Proposition \ref{prop.resol.LpLq.pm1} for details.

\subsection{Energy method}\label{subsec.apriori}

We start with a priori estimates for \eqref{eq.resol} using energy method. 
%
\begin{lemma}\label{lem.apriori.est.}
Let $\alpha,\delta\in\R$. For $\lambda\in\C$, $q\in(1,2]$ and $f\in L^2_{\sigma}(\Omega)\cap L^q(\Omega)^2$, suppose that there is a solution $v\in D(\mathbb{A}_V)$ of \eqref{eq.resol}. Then we have the following. 
\begin{enumerate}[(1)]
\item\label{item1.lem.apriori.est.1}
For $v_n$ with $|n|=1$, 
\begin{align*}
\begin{split}
&\Big(|\Im \lambda| + \Re \lambda - 4e\alpha^2 e^{-\frac{1}{4|\alpha|}} \Big) \|v_n\|_{L^2}^2 
+ \Big\{\frac14 - 2(|\alpha|+|\delta|) \Big\} \|\nabla v_n\|_{L^2}^2 \\
&\le 
C \|f\|_{L^q}^{\frac{2q}{3q-2}} 
\|v_n\|_{L^2}^{\frac{4(q-1)}{3q-2}}. 
\end{split}
\end{align*}

\item\label{item2.lem.apriori.est.1}
For $v_{\neq} = v - \sum_{|n|=1} v_n$, 
\begin{align*}
\begin{split}
&(|\Im \lambda| + \Re \lambda) \|v_{\neq}\|_{L^2}^2 
+ \Big\{\frac34 - 2(|\alpha|+|\delta|) \Big\} \|\nabla v_{\neq}\|_{L^2}^2 \\
&\le 
C \|f\|_{L^q}^{\frac{2q}{3q-2}} 
\|v_{\neq}\|_{L^2}^{\frac{4(q-1)}{3q-2}}. 
\end{split}
\end{align*}
\end{enumerate}
The constant $C$ depends only on $q$. 
\end{lemma}
%
%
\begin{proof}
(1) Taking the inner product of \eqref{eq.resol} with $v_n$, we see that 
$$
\lambda \|v_n\|_{L^2}^2 
- \langle -\mathbb{A}_{V} v_n, v_n\rangle
= \langle f, v_n \rangle 
$$
and hence that 
\begin{align}\label{est1.proof.lem.apriori.est.}
\begin{split}
(|\Im \lambda| + \Re \lambda) \|v_n\|_{L^2}^2 
\le 
|\Im \langle -\mathbb{A}_{V} v_n, v_n\rangle| 
+ \Re \langle -\mathbb{A}_{V} v_n, v_n\rangle
+ 2|\langle f, v_n\rangle|. 
\end{split}
\end{align}
From \eqref{est1.proof.prop.spec.general} and \eqref{est1.proof.prop.spec.small} in Section \ref{sec.spec.anal.}, we have 
\begin{align}\label{est2.proof.lem.apriori.est.}
\begin{split}
&|\Im \langle -\mathbb{A}_{V} v_n, v_n\rangle| 
+ \Re \langle -\mathbb{A}_{V} v_n, v_n\rangle \\
&\le 
-\|\nabla v_n\|_{L^2}^2 
+ 2|\langle V^\bot \oprot v_n, v_n\rangle| \\
&\le 
\Big\{-\frac12 + 2(|\alpha|+|\delta|) \Big\} 
\|\nabla v_n\|_{L^2}^2
+ 4e\alpha^2 e^{-\frac{1}{4|\alpha|}} 
\|v_n\|_{L^2}^2. 
\end{split}
\end{align}
One has, by the H\"{o}lder and the Gagliardo-Nirenberg inequalities, 
\begin{align}\label{est3.proof.lem.apriori.est.}
\begin{split}
|\langle f, u \rangle|
& \le \|f\|_{L^q} \|u\|_{L^{\frac{q}{q-1}}} \\
& \le C \|f\|_{L^q} \|u\|_{L^2}^{2(1-\frac1q)} 
\|\nabla u\|_{L^2}^{\frac2q-1} \\
& \le C \|f\|_{L^q}^{\frac{2q}{3q-2}} 
\|u\|_{L^2}^{\frac{4(q-1)}{3q-2}}
+ \frac18 \|\nabla u\|_{L^2}^{2}, 
\quad 
u\in W^{1,2}(\Omega)^2. 
\end{split}
\end{align}
The Young inequality is applied in the last line. The statement follows from \eqref{est1.proof.lem.apriori.est.}--\eqref{est3.proof.lem.apriori.est.}.

(2) In a similar manner as above, we see that 
\begin{align}\label{est4.proof.lem.apriori.est.}
\begin{split}
(|\Im \lambda| + \Re \lambda) \|v_{\neq}\|_{L^2}^2 
\le 
|\Im \langle -\mathbb{A}_{V} v_{\neq}, v_{\neq}\rangle| 
+ \Re \langle -\mathbb{A}_{V} v_{\neq}, v_{\neq}\rangle
+ 2|\langle f, v_{\neq}\rangle|. 
\end{split}
\end{align}
From \eqref{est1.proof.prop.spec.general} and \eqref{est1.lem.est.innerprod} in Section \ref{sec.spec.anal.}, we have 
\begin{align}\label{est5.proof.lem.apriori.est.}
\begin{split}
&|\Im \langle -\mathbb{A}_{V} v_{\neq}, v_{\neq}\rangle| 
+ \Re \langle -\mathbb{A}_{V} v_{\neq}, v_{\neq}\rangle \\ 
&\le 
-\|\nabla v_{\neq}\|_{L^2}^2 
+ 2|\langle V^\bot \oprot v_{\neq}, v_{\neq}\rangle| \\
&\le 
\{-1 + 2(|\alpha|+|\delta|)\} 
\|\nabla v_{\neq}\|_{L^2}^2. 
\end{split}
\end{align}
The statement follows from \eqref{est4.proof.lem.apriori.est.}--\eqref{est5.proof.lem.apriori.est.} combined with \eqref{est3.proof.lem.apriori.est.}. The proof is complete. 
\end{proof}
%

Lemma \ref{lem.apriori.est.} gives the following estimate of the resolvent. 
%
\begin{proposition}\label{prop.resol.est.1} 
Let $\ep\in (0,\frac{\pi}{4})$ and let $\alpha,\delta\in\R$ be sufficiently small. Set
$$
\mathcal{S}_1^\ep(\alpha)
=
\Big(
\Sigma_{\frac34 \pi - \ep} + 4e\alpha^2 e^{-\frac{1}{4|\alpha|}}
\Big)
\cap 
\Big\{|z| > 8e\alpha^2 e^{-\frac{1}{4|\alpha|}} \Big\}
\subset
\rho(-\mathbb{A}_V). 
$$
For $q\in(1,2]$ and $f\in L^2_\sigma(\Omega) \cap L^q(\Omega)^2$, we have 
\begin{equation}\label{est1.prop.resol.est.1} 
\begin{split}
\|(\lambda+{\mathbb A}_V)^{-1} f\|_{L^2} 
& \le 
C |\lambda|^{-\frac32+\frac1q} \|f\|_{L^q}, 
\quad 
\lambda\in\mathcal{S}_1^\ep(\alpha), \\
\|\nabla (\lambda+{\mathbb A}_V)^{-1} f\|_{L^2} 
& \le 
C |\lambda|^{-1+\frac1q} \|f\|_{L^q}, 
\quad 
\lambda\in\mathcal{S}_1^\ep(\alpha). 
\end{split}
\end{equation}
The constant $C$ depends only on $\ep, q$. 
\end{proposition}
%
%
\begin{proof}
Since $\mathcal{S}_1^\ep(\alpha) \subset \Sigma_{\frac34\pi-\ep} \subset \rho(-\mathbb{A}_V)$ by Corollary \ref{cor2.prop.est.Fn}, we see that $(\lambda+{\mathbb A}_V)^{-1} f$ exists for any $\lambda\in\mathcal{S}_1^\ep(\alpha)$. Observe that, if $\lambda\in\mathcal{S}_1^\ep(\alpha)$, we have both 
\begin{align*}
\begin{split}
|\Im \lambda| + \Re \lambda - 4e\alpha^2 e^{-\frac{1}{4|\alpha|}}
&=
\Big|\Im \Big(\lambda- 4e\alpha^2 e^{-\frac{1}{4|\alpha|}}\Big)\Big| 
+ \Re\Big(\lambda- 4e\alpha^2 e^{-\frac{1}{4|\alpha|}}\Big) \\
&\ge 
C \Big|\lambda - 4e\alpha^2 e^{-\frac{1}{4|\alpha|}}\Big|, 
\end{split}
\end{align*}
with a constant $C=C(\ep)$, and 
\begin{align*}
\begin{split}
\Big|\lambda - 4e\alpha^2 e^{-\frac{1}{4|\alpha|}}\Big| 
> |\lambda| - \frac{|\lambda|}{2}
= \frac{|\lambda|}{2}. 
\end{split}
\end{align*}
Hence, under the smallness on $\alpha, \delta$, Lemma \ref{lem.apriori.est.} gives 
\begin{align*}
\begin{split}
|\lambda| \|v_n\|_{L^2}^2  + \|\nabla v_n\|_{L^2}^2 
&\le 
C \|f\|_{L^q}^{\frac{2q}{3q-2}} 
\|v_n\|_{L^2}^{\frac{4(q-1)}{3q-2}}, 
\quad 
|n|=1, 
\quad 
\lambda\in\mathcal{S}_1^\ep(\alpha), \\
|\lambda| \|v_{\neq}\|_{L^2}^2  + \|\nabla v_{\neq}\|_{L^2}^2 
&\le 
C \|f\|_{L^q}^{\frac{2q}{3q-2}} 
\|v_{\neq}\|_{L^2}^{\frac{4(q-1)}{3q-2}}, 
\quad 
\lambda\in\mathcal{S}_1^\ep(\alpha)
\end{split}
\end{align*}
for the solution of \eqref{eq.resol}, namely, for $v=(\lambda+{\mathbb A}_V)^{-1} f$. This implies the assertion \eqref{est1.prop.resol.est.1}. 
\end{proof}
%

\subsection{Explicit formulas}\label{subsec.expl.form.}

Energy method can not lead to Proposition \ref{prop.resol.LpLq} due to the absence of the Hardy inequality. Instead, we employ explicit formulas and prove the following proposition.  
%
\begin{proposition}\label{prop.resol.LpLq.pm1}
Let $|n|=1$ and $\ep\in(0,\frac{\pi}4)$ and let $(\alpha,\delta)\in\R^\ast\times\R_{\ge0}$ be sufficiently small. Set 
$$
\mathcal{S}_2^\ep(\alpha)
= \Sigma_{\frac34\pi-\ep} 
\cap
\Big\{|z|<e^{-\frac{1}{4|\alpha|}}\Big\} 
\subset
\rho(-\mathbb{A}_V). 
$$
We have, for $q\in(1.2]$ and $f\in L^2_\sigma(\Omega) \cap L^q(\Omega)^2$, 
\begin{equation}\label{est1.prop.resol.LpLq.pm1}
\begin{split}
\|\mathcal{P}_n (\lambda+{\mathbb A}_{V})^{-1} f\|_{L^2} 
\le 
C |\lambda|^{-\frac32+\frac1q}
\|f\|_{L^q},
\quad 
\lambda\in \mathcal{S}_2^\ep(\alpha) 
\end{split}
\end{equation}
and, for $f\in L^2_\sigma(\Omega)$, 
\begin{equation}\label{est2.prop.resol.LpLq.pm1}
\begin{split}
\|\nabla \mathcal{P}_n (\lambda+{\mathbb A}_{V})^{-1} f\|_{L^2} 
\le 
C |\lambda|^{-\frac12}
\|f\|_{L^2},
\quad 
\lambda\in \mathcal{S}_2^\ep(\alpha). 
\end{split}
\end{equation}
The constant $C$ depends only on $\alpha,\delta,\ep,q$. 
\end{proposition}
%

The derivation of the formula is as follows. Let $\lambda\in\rho(-{\mathbb A}_{V})$ and assume first $f\in C^\infty_{0,\sigma}(\Omega)$ in \eqref{eq.resol}. Then the solution $v=(\lambda+{\mathbb A}_{V})^{-1} f$ is smooth in $\Omega$ thanks to the elliptic regularity of the Stokes system, and $\omega_n(r):=(\oprot v_n)_n(r)$ solves the equation \eqref{eq.polar.vor} in Subsection \ref{subsec.eqs.polar.coord}. Since the linearly independent solutions of its homogeneous equation are \eqref{sols.eta.K.I} in Appendix \ref{app.hom.eq.vor} and the Wronskian is $r^{-1-\delta}$, we see that $\omega_n(r)$ is given by
\begin{align}\label{formula.vor.1}
\omega_n(r) 
= 
\tilde{c}_{\lambda,n}[f_n] 
r^{-\frac{\delta}2} K_{\xi_n}(\sqrt{\lambda} r) 
+ \Phi_{\lambda,n}[f_n](r). 
\end{align}
The constant $\tilde{c}_{\lambda,n}[f_n]$ is determined later and $\Phi_{\lambda,n}[f_n]$ is defined by 
\begin{align*}
\Phi_{\lambda,n}[f_n](r)
& = 
r^{-\frac{\delta}2} K_{\xi_n}(\sqrt{\lambda} r)
\int_1^r s^{1+\frac{\delta}{2}} I_{\xi_n}(\sqrt{\lambda} s) (\oprot f_n)_n(s) \dd s \\
&\quad
+ r^{-\frac{\delta}2} I_{\xi_n}(\sqrt{\lambda} r)
\int_r^\infty s^{1+\frac{\delta}{2}} K_{\xi_n}(\sqrt{\lambda} s) (\oprot f_n)_n(s) \dd s. 
\end{align*}
Using integration by parts and setting 
\begin{align}\label{formula.g}
g^{(1)}_n
= 
\Big(\xi_n + \frac\delta2\Big) f_{\theta,n} + in f_{r,n}, 
\qquad 
g^{(2)}_n
= 
\Big(\xi_n - \frac\delta2\Big) f_{\theta,n} - in f_{r,n}, 
\end{align}
we have 
\begin{align}\label{formula.vor.2}
\begin{split}
\Phi_{\lambda,n}[f_n](r)
& =
-r^{-\frac{\delta}2} K_{\xi_n}(\sqrt{\lambda} r)
\int_1^r s^{\frac{\delta}{2}} I_{\xi_n}(\sqrt{\lambda} s) 
g^{(1)}_n(s) \dd s \\
&\quad
- \sqrt{\lambda}  r^{-\frac{\delta}2} K_{\xi_n}(\sqrt{\lambda} r)
\int_1^r s^{1+\frac{\delta}{2}} I_{\xi_n+1}(\sqrt{\lambda} s) 
f_{\theta,n}(s) \dd s \\
&\quad
+ r^{-\frac{\delta}2} I_{\xi_n}(\sqrt{\lambda} r)
\int_r^\infty 
s^{\frac{\delta}{2}} K_{\xi_n}(\sqrt{\lambda} s) 
g^{(2)}_n(s) \dd s \\
&\quad
+ \sqrt{\lambda} 
r^{-\frac{\delta}2} I_{\xi_n}(\sqrt{\lambda} r) 
\int_r^\infty s^{1+\frac{\delta}{2}} K_{\xi_n-1}(\sqrt{\lambda} s) 
f_{\theta,n}(s) \dd s. 
\end{split}
\end{align}
Since $\omega_n(r)$ decays exponentially, we see from Proposition \ref{prop.biot-savart} that $v_n$ is uniquely represented by the Biot–Savart law as, with the notations in \eqref{def.psi_n}--\eqref{def.Biot-Savart}, 
\begin{align}\label{formula.vel.1}
\begin{split}
v_n = \mathcal{V}_n[\omega_n]
= 
\tilde{c}_{\lambda,n}[f_n] 
\mathcal{V}_n\big[r\mapsto r^{-\frac{\delta}2} K_{\xi_n}(\sqrt{\lambda} r)\big]
+ \mathcal{V}_n\big[\Phi_{\lambda,n}[f_n]\big]. 
\end{split}
\end{align}
This formula is implemented with the constraint $d_n[\omega_n]=0$, which we write 
\begin{align}\label{formula.const.1}
\tilde{c}_{n,\lambda}[f_n]
F_n(\sqrt{\lambda}) 
+ d_n\big[\Phi_{n,\lambda}[f_n]\big]
= 0, 
\end{align}
by using $F_n(\sqrt{\lambda})$ in \eqref{def.Fn}. This relation determines $\tilde{c}_{n,\lambda}[f_n]$. We set 
\begin{align}\label{formula.const.2}
c_{n,\lambda}[f_n]
= 
d_n\big[\Phi_{n,\lambda}[f_n]\big] 
= 
\int_1^\infty s^{1-|n|} \Phi_{n,\lambda}[f_n](s) \dd s. 
\end{align}

Collecting \eqref{formula.vor.1}--\eqref{formula.const.2}, we find that 
\begin{align}\label{formula.vel.collected}
\begin{split}
\mathcal{P}_n (\lambda+{\mathbb A}_{V})^{-1} f
&= 
-\frac{c_{n,\lambda}[f_n]}{F_n(\sqrt{\lambda})}
\mathcal{V}_n\big[r\mapsto r^{-\frac{\delta}2} K_{\xi_n}(\sqrt{\lambda} r)\big]
+ \mathcal{V}_n\big[\Phi_{n,\lambda}[f_n]\big] 
\end{split}
\end{align}
and that 
\begin{align}\label{formula.vor.collected}
\begin{split}
&\big(\oprot\mathcal{P}_n (\lambda+{\mathbb A}_{V})^{-1} f\big)(r,\theta) 
= 
-\frac{c_{n,\lambda}[f_n]}{F_n(\sqrt{\lambda})} 
r^{-\frac{\delta}2} K_{\xi_n}(\sqrt{\lambda} r) e^{in\theta} 
+ \Phi_{n,\lambda}[f_n] e^{in\theta}. 
\end{split}
\end{align}
For general $f\in L^2_\sigma(\Omega)$, one should understand the formulas \eqref{formula.vel.collected}--\eqref{formula.vor.collected} by density argument. This understanding is possible thanks to the estimates in Proposition \ref{prop.resol.LpLq.pm1}. Note that the uniqueness of representation is guaranteed by Proposition \ref{prop.biot-savart}.

Now we let $|n|=1$ and estimate \eqref{formula.vel.collected}--\eqref{formula.vor.collected}. Firstly we estimate 
$$
\mathcal{V}_n\big[\Phi_{n,\lambda}[f_n]\big]
= \mathcal{V}_{r,n}\big[\Phi_{n,\lambda}[f_n]\big](r) e^{i n \theta}{\bf e}_r  
+  \mathcal{V}_{\theta,n}\big[\Phi_{n,\lambda}[f_n]\big](r) e^{i n\theta} {\bf e}_\theta 
$$
in \eqref{formula.vel.collected}, where 
\begin{align*}
\mathcal{V}_{r,n}\big[\Phi_{n,\lambda}[f_n]\big]
& 
=  -\frac{in}{2r} 
\bigg( \frac{c_{n,\lambda}[f_n]}{r} 
- \frac{1}{r} \int_1^r s^2 \Phi_{n,\lambda}[f_n](s) \dd s  
- r \int_r^\infty \Phi_{n,\lambda}[f_n](s) \dd s \bigg), \\
\mathcal{V}_{\theta,n}\big[\Phi_{n,\lambda}[f_n]\big]
& 
= \frac{1}{2r} 
\bigg( \frac{c_{n,\lambda}[f_n]}{r} 
- \frac{1}{r} \int_1^r s^2 \Phi_{n,\lambda}[f_n](s) \dd s   
+ r \int_r^\infty \Phi_{n,\lambda}[f_n](s) \dd s \bigg). 
\end{align*}
%

%
\begin{lemma}\label{lem.resol.LpLq.pm1.1}
Let $|n|=1$ and let $\alpha,\delta\in\R$. For $\lambda\in\C\setminus\R_{\le0}$ and $f\in C^\infty_{0,\sigma}(\Omega)$, we have 
\begin{align}
\frac{1}{r} \int_1^r s^2 \Phi_{n,\lambda}[f_n](s) \dd s 
&= 
\sum_{l=1}^{9} J_l[f_n](r), 
\label{eq1.lem.resol.LpLq.pm1.1} \\
r
\int_r^\infty 
\Phi_{n,\lambda}[f_n](s) \dd s
&= \sum_{l=10}^{17} J_l[f_n](r), 
\label{eq2.lem.resol.LpLq.pm1.1} 
\end{align}
and 
\begin{align}
c_{n,\lambda}[f_n]
&=
\sum_{l=11,13,14,15,17} J_l[f_n](1), 
\label{eq3.lem.resol.LpLq.pm1.1} 
\end{align}
where 
\begin{align*}
J_1[f_n](r) 
&= 
-\frac{1}{r} 
\int_1^r 
\tau^{\frac{\delta}2} I_{\xi_n}(\sqrt{\lambda}\tau) g^{(1)}_n(\tau) 
\int_\tau^r 
s^{2-\frac{\delta}2} K_{\xi_n}(\sqrt{\lambda} s) \dd s 
\dd \tau, \\
J_2[f_n](r) 
&= 
-\Big(\xi_n+1-\frac{\delta}2\Big) 
\frac{1}{r} 
\int_1^r 
\tau^{1+\frac{\delta}2} 
I_{\xi_n+1}(\sqrt{\lambda} \tau) f_{\theta,n}(\tau) 
\int_\tau^r 
s^{1-\frac{\delta}2} K_{\xi_n-1}(\sqrt{\lambda} s) \dd s 
\dd \tau, \\
J_3[f_n](r) 
&= 
\frac{1}{r} 
\int_1^r 
\tau^{\frac{\delta}2}  K_{\xi_n}(\sqrt{\lambda}\tau)
g^{(2)}_n(\tau) 
\int_1^\tau 
s^{2-\frac{\delta}2} I_{\xi_n}(\sqrt{\lambda} s) \dd s 
\dd \tau, \\
J_4[f_n](r) 
&= 
\Big(\xi_n-1+\frac{\delta}2\Big) 
\frac{1}{r} 
\int_1^r 
\tau^{1+\frac{\delta}2}  K_{\xi_n-1}(\sqrt{\lambda}\tau)
f_{\theta,n}(\tau)
\int_1^\tau 
s^{1-\frac{\delta}2} 
I_{\xi_n+1}(\sqrt{\lambda} s) 
\dd s 
\dd \tau, \\
J_5[f_n](r) 
&= 
\frac{1}{r} 
\int_r^\infty 
s^{\frac{\delta}2} 
K_{\xi_n}(\sqrt{\lambda}s) g^{(2)}_n(s) \dd s
\int_1^r 
s^{2-\frac{\delta}2} 
I_{\xi_n}(\sqrt{\lambda} s) \dd s, \\
J_6[f_n](r) 
&= 
\Big(\xi_n-1+\frac{\delta}2\Big) 
\frac{1}{r} 
\int_r^\infty 
s^{1+\frac{\delta}2} 
K_{\xi_n-1}(\sqrt{\lambda}s) f_{\theta,n}(s) \dd s
\int_1^r 
s^{1-\frac{\delta}2} I_{\xi_n+1}(\sqrt{\lambda} s) \dd s, \\
J_7[f_n](r) 
&= 
r^{1-\frac{\delta}2}
K_{\xi_n-1}(\sqrt{\lambda}r)
\int_1^r 
\tau^{1+\frac{\delta}2}  I_{\xi_n+1}(\sqrt{\lambda}\tau)
f_{\theta,n}(\tau)
\dd \tau, \\
J_8[f_n](r) 
&= 
r^{1-\frac{\delta}2}
I_{\xi_n+1}(\sqrt{\lambda}r)
\int_r^\infty 
\tau^{1+\frac{\delta}2} K_{\xi_n-1}(\sqrt{\lambda}\tau)
f_{\theta,n}(\tau)
\dd \tau, \\
J_9[f_n](r) 
&= 
-I_{\xi_n+1}(\sqrt{\lambda})
\frac{1}{r} 
\int_1^\infty 
\tau^{1+\frac{\delta}2} K_{\xi_n-1}(\sqrt{\lambda}\tau)
f_{\theta,n}(\tau)
\dd \tau, \\
J_{10}[f_n](r) 
&= 
-r 
\int_1^r 
s^{\frac{\delta}2} 
I_{\xi_n}(\sqrt{\lambda} s) g^{(1)}_n(s) 
\dd s
\int_r^\infty 
s^{-\frac{\delta}2} K_{\xi_n}(\sqrt{\lambda} s) 
\dd s, \\
J_{11}[f_n](r) 
&= 
-r 
\int_r^\infty 
\tau^{\frac{\delta}2} I_{\xi_n}(\sqrt{\lambda} \tau)  g^{(1)}_n(\tau) \int_\tau^\infty s^{-\frac{\delta}2} K_{\xi_n}(\sqrt{\lambda} s) 
\dd s \dd \tau, \\
J_{12}[f_n](r) 
&= 
-\Big(\xi_n-1-\frac{\delta}2\Big)
r 
\int_1^r 
s^{1+\frac{\delta}2} 
I_{\xi_n+1}(\sqrt{\lambda} s) 
f_{\theta,n}(s)
\dd s
\int_r^\infty 
s^{-1-\frac{\delta}2} K_{\xi_n-1}(\sqrt{\lambda} s) 
\dd s, \\
J_{13}[f_n](r) 
&= 
-\Big(\xi_n-1-\frac{\delta}2\Big)
r 
\int_r^\infty 
\tau^{1+\frac{\delta}2} 
I_{\xi_n+1}(\sqrt{\lambda} \tau) 
f_{\theta,n}(s) 
\int_\tau^\infty 
s^{-1-\frac{\delta}2} 
K_{\xi_n-1}(\sqrt{\lambda} s) 
\dd s \dd \tau, \\
J_{14}[f_n](r) 
&= 
r 
\int_r^\infty 
\tau^{\frac{\delta}2} K_{\xi_n}(\sqrt{\lambda} \tau) g^{(2)}_n(\tau) 
\int_r^\tau 
s^{-\frac{\delta}2} I_{\xi_n}(\sqrt{\lambda} s) 
\dd s \dd \tau, \\
J_{15}[f_n](r) 
&= 
\Big(\xi_n + 1 + \frac{\delta}2\Big)
r 
\int_r^\infty 
\tau^{1+\frac{\delta}2} 
K_{\xi_n-1}(\sqrt{\lambda} \tau) 
f_{\theta,n}(\tau) 
\int_r^\tau 
s^{-1-\frac{\delta}2} I_{\xi_n+1}(\sqrt{\lambda} s) 
\dd s \dd \tau, \\
J_{16}[f_n](r) 
&= 
-r^{1-\frac{\delta}2}
K_{\xi_n-1}(\sqrt{\lambda} r)
\int_1^r 
\tau^{1+\frac{\delta}2} I_{\xi_n+1}(\sqrt{\lambda} \tau) f_{\theta,n}(\tau) 
\dd \tau, \\
J_{17}[f_n](r) 
&= 
-r^{1-\frac{\delta}2}
I_{\xi_n+1}(\sqrt{\lambda} r)
\int_r^\infty 
\tau^{1+\frac{\delta}2} K_{\xi_n-1}(\sqrt{\lambda} \tau)  f_{\theta,n}(\tau) 
\dd \tau. 
\end{align*}
\end{lemma}
%
%
\begin{remark}\label{rem.lem.decom.formula.resol}
\begin{enumerate}[(1)]
\item\label{item1.rem.lem.decom.formula.resol}
From \eqref{eq1.lem.resol.LpLq.pm1.1}--\eqref{eq3.lem.resol.LpLq.pm1.1}, we see that 
$$
\frac{c_{n,\lambda}[f_n]}{r} 
- \frac{1}{r} \int_1^r s^2 \Phi_{n,\lambda}[f_n](s) \dd s 
= 
\frac1r
\sum_{l=11,13,14,15} J_l[f_n](1) 
- \sum_{l=1}^{8} J_l[f_n](r). 
$$
Thus we do not take the term $J_9[f_n]$ into account when estimating $\mathcal{V}_n\big[\Phi_{n,\lambda}[f_n]\big]$.

\item\label{item2.rem.lem.decom.formula.resol}
Observe that $J_{7}[f_n]=-J_{16}[f_n]$ and that $J_{8}[f_n]=-J_{17}[f_n]$. 
\end{enumerate}
\end{remark}
%
%
\begin{proofx}{Lemma \ref{lem.resol.LpLq.pm1.1}}
The equalities \eqref{eq1.lem.resol.LpLq.pm1.1}--\eqref{eq2.lem.resol.LpLq.pm1.1} can be proved by change of the order of integration, the recurrence relations (see \cite[Chapter $\mathrm{I}\hspace{-1.2pt}\mathrm{I}\hspace{-1.2pt}\mathrm{I}$ 3$\cdot$71 (3), (4)]{Watson(1944)})
\begin{align*}
z K_{\mu}(z)
&= 
(\mu-1) K_{\mu-1}(z) 
- z \frac{\dd K_{\mu-1}}{\dd z}(z), \\
z I_{\mu}(z)
&= 
(\mu+1) I_{\mu+1}(z) 
+ z \frac{\dd I_{\mu+1}}{\dd z}(z), 
\end{align*}
and integration by parts. We omit the details since they are analogous to those in the proof of \cite[Lemmas 3.6 and 3.9]{Maekawa(2017a)} corresponding to the case $\delta=0$. The equality \eqref{eq3.lem.resol.LpLq.pm1.1} follows from the definition \eqref{formula.const.2} and \eqref{eq2.lem.resol.LpLq.pm1.1} with $r=1$. The proof is complete. 
\end{proofx}
%

%
\begin{lemma}\label{lem.resol.LpLq.pm1.2}
Let $|n|=1$ and $\ep\in(0,\pi)$ and let $(\alpha,\delta)\in\R^\ast\times\R_{\ge0}$. We have the following. 
\begin{enumerate}[(1)]
\item\label{item1.lem.resol.LpLq.pm1.2}
Let $l\in\{1,\ldots,17\}\setminus\{7,8,9,16,17\}$. For $q\in[1.\infty)$ and $f\in C^\infty_{0,\sigma}(\Omega)$, 
\begin{align*}
\sup_{r\ge1} 
r^{\frac2q} 
|r^{-1} J_l[f_n](r)|
&\le 
C |\lambda|^{-1} \|f_n\|_{L^q}, 
\quad 
\lambda\in\Sigma_{\pi-\ep} \cap \{|z|<1\}, 
\\
\sup_{r\ge1} 
|r^{-1} J_l[f_n](r)|
&\le C |\lambda|^{-1} \|f_n\|_{L^1}, 
\quad 
\lambda\in\Sigma_{\pi-\ep} \cap \{|z|<1\}. 
\end{align*}

\item\label{item2.lem.resol.LpLq.pm1.2}
Let $l\in\{7,8,16,17\}$. For $f\in C^\infty_{0,\sigma}(\Omega)$, 
\begin{align*}
\int_{1}^{\infty}
|r^{-1} J_l[f_n](r)| r \dd r 
&\le 
C |\lambda|^{-1} \|f_n\|_{L^1}, 
\quad 
\lambda\in\Sigma_{\pi-\ep} \cap \{|z|<1\}, \\
\sup_{r\ge1} 
|r^{-1} J_l[f_n](r)|
&\le 
C |\lambda|^{-1} \|f_n\|_{L^\infty}, 
\quad 
\lambda\in\Sigma_{\pi-\ep} \cap \{|z|<1\}, \\
\sup_{r\ge1} 
|r^{-1} J_l[f_n](r)|
&\le C \|f_n\|_{L^1}, 
\quad 
\lambda\in\Sigma_{\pi-\ep} \cap \{|z|<1\}. 
\end{align*}
\end{enumerate}
The constant $C$ depends only on $\alpha,\delta,\ep,q$. 
\end{lemma}
%
%
\begin{proof}
Each of the estimates can be proved by Lemmas \ref{lem.est.int.vel.K} and \ref{lem.est.int.vel.I} in Appendix \ref{appendix.bessel}. We omit the calculations since they are analogous to the ones in the proof of \cite[Lemmas 3.7 and 3.10]{Maekawa(2017a)} corresponding to the case $\delta=0$. The proof is complete. 
\end{proof}
%

%
\begin{lemma}\label{lem.resol.LpLq.pm1.3}
Let $|n|=1$ and $\ep\in(0,\pi)$ and let $(\alpha,\delta)\in\R^\ast\times\R_{\ge0}$. 
We have the following. 
\begin{enumerate}[(1)]
\item\label{item1.lem.resol.LpLq.pm1.3}
Let $l\in\{1,\ldots,17\}\setminus\{9\}$. For $1\le q<p\le\infty$ or $1<q\le p<\infty$ and $f\in C^\infty_{0,\sigma}(\Omega)$, 
\begin{align*}
\big\|(r,\theta) 
\mapsto 
r^{-1} J_l[f_n](r)
\big\|_{L^p}
&\le
C |\lambda|^{-1+\frac1q-\frac1p} \|f \|_{L^q}, 
\quad 
\lambda\in\Sigma_{\pi-\ep} \cap \{|z|<1\}. 
\end{align*}

\item\label{item2.lem.resol.LpLq.pm1.3}
For $q\in(1.\infty)$ and $f\in C^\infty_{0,\sigma}(\Omega)$, 
\begin{align*}
|c_{n,\lambda}[f_n]| 
&\le
C |\lambda|^{-1+\frac1q} \|f \|_{L^q}, 
\quad 
\lambda\in\Sigma_{\pi-\ep} \cap \{|z|<1\}. 
\end{align*}
\end{enumerate}
The constant $C$ depends only on $\alpha,\delta,\ep,q,p$. 
\end{lemma}
%
%
\begin{proof}
(1) The estimate can be proved by Lemma \ref{lem.resol.LpLq.pm1.2} and interpolation theorems. We omit the details since they are analogous to those in the proof of \cite[Corollary 3.12]{Maekawa(2017a)} and \cite[Corollary 3.8]{Higaki(2019)} corresponding to the case $\delta=0$.

(2) The estimate follows from \eqref{eq3.lem.resol.LpLq.pm1.1} and (\ref{item1.lem.resol.LpLq.pm1.3}) with $p=\infty$. The proof is complete. 
\end{proof}
%

Next we estimate $\mathcal{V}_n\big[r\mapsto r^{-\frac{\delta}2} K_{\xi_n}(\sqrt{\lambda} r)\big]$ in \eqref{formula.vel.collected} and the terms in \eqref{formula.vor.collected}. 
%
\begin{lemma}\label{lem.resol.LpLq.pm1.vor}
Let $|n|=1$ and $\ep\in(0,\pi)$ and let $(\alpha,\delta)\in\R^\ast\times\R_{\ge0}$. 
We have the following. 
\begin{enumerate}[(1)]
\item\label{item1.lem.resol.LpLq.pm1.vor}
For $p\in(1,\infty]$, 
$$
\big\|
\mathcal{V}_n\big[r\mapsto r^{-\frac{\delta}2} K_{\xi_n}(\sqrt{\lambda} r)\big]
\big\|_{L^p}
\le
C |\lambda|^{-\frac{\Re \xi_n}{2}-\frac1p}, 
\quad 
\lambda\in\Sigma_{\pi-\ep} \cap \{|z|<1\}. 
$$

\item\label{item2.lem.resol.LpLq.pm1.vor}
For $p\in[1,2]$, 
$$
\big\|(r,\theta) 
\mapsto 
r^{-\frac{\delta}2} K_{\xi_n}(\sqrt{\lambda} r) e^{in\theta} 
\big\|_{L^p}
\le C |\lambda|^{-\frac{\Re \xi_n}{2}-\frac1p+\frac12}, 
\quad 
\lambda\in\Sigma_{\pi-\ep} \cap \{|z|<1\}, 
$$
and for $p\in[2,\infty)$, 
$$
\big\|(r,\theta) 
\mapsto 
r^{-\frac{\delta}2} K_{\xi_n}(\sqrt{\lambda} r) e^{in\theta} 
\big\|_{L^p}
\le C |\lambda|^{-\frac{\Re \xi_n}{2}}, 
\quad 
\lambda\in\Sigma_{\pi-\ep} \cap \{|z|<1\}. 
$$

\item\label{item3.lem.resol.LpLq.pm1.vor}
For $p\in[1,\infty]$ and $f\in C^\infty_{0,\sigma}(\Omega)$, 
$$
\big\|(r,\theta) 
\mapsto 
\Phi_{n,\lambda}[f_n](r) 
e^{in\theta}
\big\|_{L^p}
\le 
C |\lambda|^{-\frac12} \|f_n\|_{L^p}, 
\quad 
\lambda\in\Sigma_{\pi-\ep} \cap \{|z|<1\}. 
$$
\end{enumerate}
The constant $C$ depends only on $\alpha,\delta,\ep,p$. 
\end{lemma}
%
%
\begin{proof}
(1) The estimate can be proved by Lemma \ref{lem.est.int.vel.K} and interpolation theorems. We omit the details since they are analogous to those in the proof of \cite[Proposition 3.17]{Maekawa(2017a)} and \cite[Proposition 3.9]{Higaki(2019)} corresponding to the case $\delta=0$.

(2) The estimate follows from the inequality 
$$
\big\|(r,\theta) 
\mapsto 
r^{-\frac{\delta}2} K_{\xi_n}(\sqrt{\lambda} r) e^{in\theta} 
\big\|_{L^p}
\le 
\big\|(r,\theta) 
\mapsto 
K_{\xi_n}(\sqrt{\lambda} r) e^{in\theta} 
\big\|_{L^p}. 
$$
and the estimate of the right-hand in \cite[Lemma 3.22]{Maekawa(2017a)} and \cite[Lemma B.4]{Higaki(2019)}.

(3) The estimate can be proved by Lemma \ref{lem.est.int.vor} and interpolation theorems. We omit the details since they are analogous to those in the proof of \cite[Lemma 3.21]{Maekawa(2017a)} corresponding to the case $\delta=0$. The proof is complete. 
\end{proof}
%

%
\begin{proofx}{Proposition \ref{prop.resol.LpLq.pm1}}
Since $\mathcal{S}_2^\ep(\alpha)\subset \Sigma_{\frac34\pi-\ep}\subset \rho(-\mathbb{A}_V)$ by Corollary \ref{cor2.prop.est.Fn}, we see that $(\lambda+{\mathbb A}_V)^{-1} f$ exists for any $\lambda\in\mathcal{S}_2^\ep(\alpha)$. Let $\lambda\in \mathcal{S}_2^\ep(\alpha)$. By density argument, it suffices to prove \eqref{est1.prop.resol.LpLq.pm1}--\eqref{est2.prop.resol.LpLq.pm1} for $f\in C^\infty_{0,\sigma}(\Omega)$. From Corollaries \ref{cor1.prop.est.Fn} and \ref{lem.resol.LpLq.pm1.3} (\ref{item2.lem.resol.LpLq.pm1.3}), we have 
$$
\Big|
\frac{c_{n,\lambda}[f_n]}{F_n(\sqrt{\lambda})} 
\Big|
\le 
C 
|\lambda|^{-1+\frac1q+\frac{\Re\xi_n}{2}} 
\|f \|_{L^q}. 
$$
Thus, from \eqref{formula.vel.collected} and Lemma \ref{lem.resol.LpLq.pm1.1}, putting $p=2$ in Lemmas \ref{lem.resol.LpLq.pm1.3} and \ref{lem.resol.LpLq.pm1.vor}, we see that 
\begin{equation*}
\begin{split}
\|\mathcal{P}_n (\lambda+{\mathbb A}_{V})^{-1} f\|_{L^2} 
\le 
C |\lambda|^{-\frac32+\frac1q}
\|f\|_{L^q}, 
\end{split}
\end{equation*}
which is \eqref{est1.prop.resol.LpLq.pm1}. Also, from \eqref{formula.vor.collected}, putting $p=2$ in Lemma \ref{lem.resol.LpLq.pm1.vor}, we see that 
\begin{equation*}
\begin{split}
\|\oprot \mathcal{P}_n (\lambda+{\mathbb A}_{V})^{-1} f\|_{L^2} 
\le 
C |\lambda|^{-\frac12} 
\|f\|_{L^2}, 
\end{split}
\end{equation*}
which leads to \eqref{est2.prop.resol.LpLq.pm1} since $\|\oprot u\|_{L^2}=\|\nabla u\|_{L^2}$ for $u\in W^{1,2}_0(\Omega)^2\cap L^2_\sigma(\Omega)$. All the constants $C$ above are independent of $\lambda$. This completes the proof. 
\end{proofx}
%

\subsection{Proof of Proposition \ref{prop.resol.LpLq}}\label{subsec.proof.prop.resol.LpLq}

Proposition \ref{prop.resol.LpLq} is a consequence of Lemma \ref{lem.apriori.est.} and Propositions \ref{prop.resol.est.1} and \ref{prop.resol.LpLq.pm1}. 
%
\begin{proofx}{Proposition \ref{prop.resol.LpLq}}
Let $\ep\in(0,\frac{\pi}4)$ be given. The same consideration as in the proof of Corollary \ref{cor2.prop.est.Fn} shows that $\Sigma_{\frac34\pi-\ep}\subset\mathcal{S}_1^\frac{\ep}2(\alpha)\cup\mathcal{S}_2^\frac{\ep}2(\alpha)$ 
for sufficiently small $\alpha,\delta$ depending on $\ep$. In view of Proposition \ref{prop.resol.est.1}, the desired estimates \eqref{est1.prop.resol.LpLq}--\eqref{est2.prop.resol.LpLq} follow if we prove 
\begin{equation}\label{est1.proof.prop.resol.LpLq}
\begin{split}
\|(\lambda+{\mathbb A}_V)^{-1} f\|_{L^2} 
& \le 
C |\lambda|^{-\frac32+\frac1q} \|f\|_{L^q}, 
\quad 
\lambda\in\mathcal{S}_2^\ep(\alpha), \\
\|\nabla (\lambda+{\mathbb A}_V)^{-1} f\|_{L^2} 
& \le 
C |\lambda|^{-1+\frac1q} \|f\|_{L^2}, 
\quad 
\lambda\in\mathcal{S}_2^\ep(\alpha) 
\end{split}
\end{equation}
for $f\in C^\infty_{0,\sigma}(\Omega)$ and $\ep\in(0,\frac{\pi}4)$. Let $\lambda\in\mathcal{S}_2^\ep(\alpha)$ and set $v=(\lambda+{\mathbb A}_V)^{-1} f$. Thanks to Proposition \ref{prop.resol.LpLq.pm1}, we only need to estimate $v_{\neq} = v - \sum_{|n|=1} v_n$. Lemma \ref{lem.apriori.est.} (\ref{item2.lem.apriori.est.1}) implies that 
\begin{align*}
\begin{split}
|\lambda| \|v_{\neq}\|_{L^2}^2  + \|\nabla v_{\neq}\|_{L^2}^2 
&\le 
C \|f\|_{L^q}^{\frac{2q}{3q-2}} 
\|v_{\neq}\|_{L^2}^{\frac{4(q-1)}{3q-2}}, 
\quad 
\lambda\in\mathcal{S}_2^\ep(\alpha)
\end{split}
\end{align*}
with a constant $C=C(\ep)$, and hence that 
\begin{equation*}
\begin{split}
\|v_{\neq}\|_{L^2} 
& \le 
C |\lambda|^{-\frac32+\frac1q} \|f\|_{L^q}, 
\quad 
\lambda\in\mathcal{S}_2^\ep(\alpha), \\
\|\nabla v_{\neq}\|_{L^2} 
& \le 
C |\lambda|^{-1+\frac1q} \|f\|_{L^q}, 
\quad 
\lambda\in\mathcal{S}_2^\ep(\alpha). 
\end{split}
\end{equation*}
Hence the proof is complete. 
\end{proofx}
%

\appendix

\section{Modified Bessel function}\label{appendix.bessel}

We summarize the facts about the modified Bessel functions. Our main references are \cite{Watson(1944), Andrews-Askey-Roy(1999)}. The modified Bessel function of the first kind $I_\mu(z)$ of order $\mu$ is defined by 
\begin{align}\label{def.I}
I_\mu(z) 
= 
\Big(\frac{z}{2}\Big)^\mu 
\sum_{m=0}^{\infty} 
\frac{1}{m!\Gamma(\mu+m+1)} \Big(\frac{z}{2}\Big)^{2m}, 
\quad 
z\in \C\setminus \R_{\le0}, 
\end{align}
where $\Gamma(z)$ is the Gamma function, the second kind $K_\mu(z)$ of order $\mu\notin\Z$ is by 
\begin{align}\label{def.K}
K_\mu(z) 
= 
\frac{\pi}{2}
\frac{I_{-\mu}(z) - I_\mu(z)}{\sin(\mu \pi)}, 
\quad 
z \in \C\setminus\R_{\le0}, 
\end{align}
and $K_n(z)$ of order $n\in\Z$ is by the limit of $K_{\mu}(z)$ in \eqref{def.K} as $\mu\to n$. In this paper, we exclusively consider the case where the order $\mu$ satisfies $\mu\notin\Z$ and $\Re\mu>0$.

The functions $K_\mu(z)$ and $I_\mu(z)$ are linearly independent solutions of 
$$
-\frac{\dd^2 \omega}{\dd z^2} 
- \frac{1}{z} \frac{\dd \omega}{\dd z} 
+ \Big(1+\frac{\mu^2}{z^2} \Big) \omega = 0, 
\quad 
z\in \C\setminus \R_{\le0}, 
$$
with the Wronskian 
\begin{align}\label{def.Wronskian}
\det 
\left(
\begin{array}{cc}
K_\mu(z) 
& I_\mu(z) \\
\displaystyle{\frac{\dd K_\mu}{\dd z}(z)} 
& \displaystyle{\frac{\dd I_\mu}{\dd z}(z)}
\end{array}
\right)
=
\frac1z. 
\end{align}
It is well known that $I_\mu(z)$ grows exponentially and $K_\mu(z)$ decays exponentially as $|z|\to\infty$ in $\Sigma_{\frac{\pi}2}$; see \cite[Section 4.12]{Andrews-Askey-Roy(1999)}. As an integral representation useful in Section \ref{sec.quant.anal.disc.spec}, we have 
\begin{align}\label{rep.K.int}
K_{\mu}(z) 
= \frac12 \int_{0}^{\infty} 
e^{-\frac{z}2(t+\frac1{t})} t^{-\mu-1} \dd t, 
\quad 
z\in \Sigma_{\frac{\pi}2}, 
\end{align}
which can be deduced by the formula \cite[Chapter $\mathrm{V}\hspace{-1.2pt}\mathrm{I}$ 6$\cdot$22 (5)]{Watson(1944)} and change of variables.

Collected below are the estimates involving $K_\mu(z)$ and $I_\mu(z)$ used in this paper. Each of them can be found in the references \cite{Andrews-Askey-Roy(1999), Watson(1944)} or follows from a simple calculation using Lemma \ref{lem.est.bessel1}. and hence we omit the proof. For the details when $\delta=0$, we refer to \cite[Lemma 3.31 and Appendix A]{Maekawa(2017a)} and to \cite{Higaki(2019)} studying the dependence on $\alpha$ in the estimates.

We recall that the constants $\eta_n$ and $\xi_n$ are defined in \eqref{def.1delta.eta} and \eqref{def.xi}, respectively.

%
\begin{lemma}\label{lem.est.bessel1}
Let $\Re\mu>0$, $\ep\in(0,\frac{\pi}{2})$ and $M>0$. We have 
\begin{align*}
|K_{\mu}(z)|
&\le 
C |z|^{-\Re\mu},
\quad 
z \in\Sigma_{\frac{\pi}2} \cap \{|z|< M\}, \\
|K_{\mu}(z)|
&\le C |z|^{-\frac12} e^{-\Re z},
\quad 
z \in\Sigma_{\frac{\pi}2} \cap \{|z|\ge M\}, \\
|I_{\mu}(z)| 
&\le 
C |z|^{\Re\mu}, 
\quad 
z \in\Sigma_{\frac{\pi}2} \cap \{|z|< M\}, \\
|I_{\mu}(z)|
&\le C |z|^{-\frac12} e^{\Re z},
\quad 
z \in\Sigma_{\frac{\pi}2-\ep} \cap \{|z|\ge M\}. 
\end{align*}
The constant $C$ depends on $\mu,\ep,M$. 
\end{lemma}
%

%
\begin{lemma}\label{lem.est.bessel2}
Let $|n|=1$. We have the following. 
\begin{enumerate}[(1)]
\item\label{item1.lem.est.bessel2}
For sufficiently small $\alpha,\delta\in\R$,
$$
K_{1+\eta_n}(z) 
= 
\frac{\Gamma(1+\eta_n)}{2} 
\Big(\frac{z}{2}\Big)^{-1-\eta_n}
+ R^{(1)}_n(z),
\quad 
z \in\Sigma_{\frac{\pi}2} \cap \{|z|<1\}.
$$
Here $R^{(1)}_n(z)$ is the remainder and satisfies
$$
|R^{(1)}_n(z)| 
\le 
C|z|^{1-\Re \eta_n} 
\big(1 + \big|\log|z|\big|\big), 
\quad 
z \in\Sigma_{\frac{\pi}2} \cap \{|z|<1\}.
$$

\item\label{item2.lem.est.bessel2}
For sufficiently small $\alpha,\delta\in\R$, 
\begin{align*}
K_{\eta_n}(z) 
= 
\frac{\pi}{2\sin(\eta_n\pi)}
\bigg(
\frac{1}{\Gamma(1-\eta_n)} \Big(\frac{z}{2}\Big)^{-\eta_n}
- \frac{1}{\Gamma(1+\eta_n)} \Big(\frac{z}{2}\Big)^{\eta_n}
\bigg) 
+ \tilde{R}^{(1)}_n(z), & \\
\quad 
z \in\Sigma_{\frac{\pi}2} \cap \{|z|<1\}.&
\end{align*}
Here $\tilde{R}^{(1)}_n(z)$ is the remainder and satisfies
$$
|\tilde{R}^{(1)}_n(z)| 
\le C|z|^{2-\Re\eta_n}
\big(1 + \big|\log|z|\big|\big), 
\quad 
z \in\Sigma_{\frac{\pi}2} \cap \{|z|<1\}.
$$
\end{enumerate}
The constant $C$ is independent of $\alpha,\delta$. 
\end{lemma}
%

%
\begin{lemma}\label{lem.est.int.vel.K} 
Let $|n|=1$ and $\ep\in(0,\pi)$ and let $(\alpha,\delta)\in\R^\ast\times\R_{\ge0}$. For $\lambda\in\Sigma_{\pi-\ep} \cap \{|z|<1\}$, we have the following. 
\begin{enumerate}[(1)]
\item\label{item1.lem.est.int.K}
For $1\le \tau \le r \le (\Re\sqrt{\lambda})^{-1}$ and $k=0,1$,
$$
\int_\tau^r 
s^{2-k-\frac{\delta}{2}} 
|K_{\xi_n-k}(\sqrt{\lambda} s)| 
\dd s
\le 
C |\lambda|^{-\frac{\Re\xi_n}2+\frac{k}2} r^{3-\Re\xi_n-\frac{\delta}2}. 
$$

\item\label{item2.lem.est.int.K}
For $1 \le \tau\le (\Re\sqrt{\lambda})^{-1} \le r \le \infty$ and $k=0,1$, 
$$
\int_\tau^r 
s^{2-k-\frac{\delta}{2}} 
|K_{\xi_n-k}(\sqrt{\lambda} s)| 
\dd s
\le 
C |\lambda|^{-\frac32+\frac{k}{2}+\frac{\delta}4}. 
$$

\item\label{item3.lem.est.int.K}
For $(\Re\sqrt{\lambda})^{-1} \le \tau\le r \le \infty$ and $k=0,1$, 
$$
\int_\tau^r 
s^{2-k-\frac{\delta}{2}} 
|K_{\xi_n-k}(\sqrt{\lambda} s)| 
\dd s
\le 
C|\lambda|^{-\frac34} \tau^{\frac32-k-\frac{\delta}2} 
e^{-(\Re\sqrt{\lambda})\tau}.
$$

\item\label{item4.lem.est.int.K}
For $1\le \tau \le (\Re\sqrt{\lambda})^{-1}$ and $k=0,1$, 
$$
\int_\tau^\infty 
s^{-k-\frac{\delta}{2}} 
|K_{\xi_n-k}(\sqrt{\lambda} s)| 
\dd s
\le 
C
|\lambda|^{-\frac{\Re\xi_n}{2}+\frac{k}{2}} 
\tau^{1-\Re\xi_n-\frac{\delta}2}.
$$

\item\label{item5.lem.est.int.K}
For $\tau \ge (\Re\sqrt{\lambda})^{-1}$ and $k=0,1$, 
$$
\int_\tau^\infty 
s^{-k-\frac{\delta}{2}} 
|K_{\xi_n-k}(\sqrt{\lambda} s)| 
\dd s
\le 
C 
|\lambda|^{-\frac34} \tau^{-\frac12-k-\frac{\delta}{2}} 
e^{-(\Re\sqrt{\lambda})\tau}.
$$
\end{enumerate}
The constant $C$ depends only on $\alpha,\delta,\ep$. 
\end{lemma}
%

%
\begin{lemma}\label{lem.est.int.vel.I}
Let $|n|=1$ and $\ep\in(0,\pi)$ and let $(\alpha,\delta)\in\R^\ast\times\R_{\ge0}$. For $\lambda\in\Sigma_{\pi-\ep} \cap \{|z|<1\}$, we have the following.
\begin{enumerate}[(1)]
\item\label{item1.lem.est.int.I}
For $1\le \tau \le (\Re\sqrt{\lambda})^{-1}$ and $k=0,1$, 
$$
\int_1^\tau 
s^{2-k-\frac{\delta}{2}} 
|I_{\xi_n+k}(\sqrt{\lambda} s)| 
\dd s
\le 
C |\lambda|^{\frac{\Re\xi_n}2+\frac{k}{2}} 
\tau^{3+\Re\xi_n-\frac{\delta}{2}}.
$$

\item\label{item2.lem.est.int.I}
For $\tau \ge (\Re\sqrt{\lambda})^{-1}$ and $k=0,1$, 
$$
\int_1^\tau 
s^{2-k-\frac{\delta}{2}} 
|I_{\xi_n+k}(\sqrt{\lambda} s)| 
\dd s
\le
C |\lambda|^{-\frac34} \tau^{\frac32-k-\frac{\delta}2} 
e^{(\Re\sqrt{\lambda})\tau}. 
$$

\item\label{item3.lem.est.int.I}
For $1\le r \le \tau \le (\Re\sqrt{\lambda})^{-1}$ and $k=0,1$, 
$$
\int_r^\tau 
s^{-k-\frac{\delta}{2}} 
|I_{\xi_n+k}(\sqrt{\lambda} s)| 
\dd s
\le
C |\lambda|^{\frac{\Re\xi_n}2+\frac{k}{2}}
\tau^{1+\Re\xi_n-\frac{\delta}{2}}.
$$

\item\label{item4.lem.est.int.I}
For $1\le r \le (\Re\sqrt{\lambda})^{-1} \le \tau$ and $k=0,1$, 
$$
\int_r^\tau 
s^{-k-\frac{\delta}{2}}  |I_{\xi_n+k}(\sqrt{\lambda} s)| 
\dd s
\le
C |\lambda|^{-\frac34} \tau^{-\frac12-k-\frac{\delta}2} 
e^{(\Re\sqrt{\lambda})\tau}.
$$

\item\label{item5.lem.est.int.I}
For $(\Re\sqrt{\lambda})^{-1} \le r \le \tau$ and $k=0,1$, 
$$
\int_r^\tau 
s^{-k-\frac{\delta}2}  |I_{\xi_n+k}(\sqrt{\lambda} s)| 
\dd s
\le
C |\lambda|^{-\frac34} \tau^{-\frac12-k-\frac{\delta}2} 
e^{(\Re\sqrt{\lambda})\tau}. 
$$
\end{enumerate}
The constant $C$ depends only on $\alpha,\delta,\ep$. 
\end{lemma}
%

%
\begin{lemma}\label{lem.est.int.vor}
Let $|n|=1$ and $\ep\in(0,\pi)$ and let $(\alpha,\delta)\in\R^\ast\times\R_{\ge0}$. For $\lambda\in\Sigma_{\pi-\ep} \cap \{|z|<1\}$, we have the following.
\begin{enumerate}[(1)]
\item\label{item1.}
For $1\le \tau \le (\Re\sqrt{\lambda})^{-1}$, 
$$
\int_1^\tau 
s^{1-\frac{\delta}{2}} 
|I_{\xi_n}(\sqrt{\lambda} s)| 
\dd s
\le 
C |\lambda|^{\frac{\Re\xi_n}2} 
\tau^{2+\Re\xi_n-\frac{\delta}{2}}.
$$

\item\label{item2.}
For $\tau \ge (\Re\sqrt{\lambda})^{-1}$,
$$
\int_1^\tau
s^{1-\frac{\delta}{2}} 
|I_{\xi_n}(\sqrt{\lambda} s)| 
\dd s
\le
C |\lambda|^{-\frac34} 
\tau^{\frac12-\frac{\delta}{2}} 
e^{(\Re\sqrt{\lambda})\tau}.
$$

\item\label{item3.}
For $1 \le \tau\le (\Re\sqrt{\lambda})^{-1}$, 
$$
\int_\tau^\infty
s^{1-\frac{\delta}{2}} 
|K_{\xi_n}(\sqrt{\lambda} s)| 
\dd s
\le 
C
|\lambda|^{-\frac{\Re\xi_n}{2}-\frac12} 
\tau^{1-\Re\xi_n-\frac{\delta}2} 
+ C 
|\lambda|^{-1+\frac{\delta}4}.
$$

\item\label{item4.}
For $\tau \ge (\Re\sqrt{\lambda})^{-1}$, 
$$
\int_\tau^\infty 
s^{1-\frac{\delta}{2}} 
|K_{\xi_n}(\sqrt{\lambda} s)| 
\dd s
\le 
C |\lambda|^{-\frac34} 
\tau^{\frac12-\frac{\delta}{2}} 
e^{-(\Re\sqrt{\lambda})\tau}.
$$
\end{enumerate}
The constant $C$ depends only on $\alpha,\delta,\ep$. 
\end{lemma}
%

\section{Homogeneous equation for vorticity}\label{app.hom.eq.vor}

For $\lambda\in\C\setminus\R_{\le0}$, we consider the homogeneous equation of \eqref{eq.polar.vor} 
$$
- \frac{\dd^2 \omega_n}{\dd r^2} 
- \frac{1+\delta}{r} \frac{\dd \omega_n}{\dd r} 
+ \Big(\lambda + \frac{n^2+i\alpha n}{r^2} \Big) \omega_n 
= 0, 
\quad 
r>1. 
$$
We will prove that its linearly independent solutions are, with $\xi_n$ defined in \eqref{def.xi}, 
\begin{align}\label{sols.eta.K.I}
r^{-\frac{\delta}2} K_{\xi_n}(\sqrt{\lambda} r)
\quad \text{and} \quad  
r^{-\frac{\delta}2} I_{\xi_n}(\sqrt{\lambda} r), 
\end{align}
and the Wronskian is $r^{-1-\delta}$. The proof is as follows. Applying the transformation 
\begin{align}\label{def.tx}
\omega_n(r) = r^{-\frac{\delta}2} \tilde{\omega}_n(r), 
\end{align}
we find that $\tilde{\omega}_n$ solves 
$$
-\frac{\dd^2 \tilde{\omega}_n}{\dd r^2} 
- \frac{1}{r} \frac{\dd \tilde{\omega}_n}{\dd r} 
+ \Big(
\lambda + \frac{\xi_n^2}{r^2} 
\Big)
\tilde{\omega}_n
= 0, \quad r>1. 
$$
By Appendix \ref{appendix.bessel}, its linearly independent solutions are 
$$
K_{\xi_n}(\sqrt{\lambda} r)
\quad \text{and} \quad 
I_{\xi_n}(\sqrt{\lambda} r). 
$$
Hence, by the inverse transformation of \eqref{def.tx}, we see that the desired solutions are \eqref{sols.eta.K.I}. One can easily compute the Wronskian using \eqref{def.Wronskian}. The proof is complete.

\section{Proof of Theorem \ref{thm.L2L2}}\label{app.proof.thm.L2L2}

Let $\ep\in(0,\frac{\pi}4)$ and fix a number $\phi\in(\frac{\pi}2,\frac{3}4\pi-\ep)$. Taking $b\in(0,1)$ and a curve $\gamma_b$ in $\C$
$$
\gamma_b 
= \{|\oparg z|=\phi, \mkern9mu |z|\ge b \} \cup \{|\oparg z|\le\phi, \mkern9mu |z|=b \}
$$
oriented counterclockwise, we use a representation of $\{e^{-t {\mathbb A}_V}\}_{t\ge0}$ in the Dunford integral 
$$
e^{-t {\mathbb A}_V} 
=
\frac{1}{2\pi i}
\int_{\gamma_b}
e^{t \lambda }(\lambda+{\mathbb A}_V)^{-1} 
\dd \lambda,
\quad 
t>0. 
$$
From \eqref{est1.prop.resol.LpLq} for $q=2$ in Proposition \ref{prop.resol.LpLq}, we see that $\{e^{-t {\mathbb A}_V}\}_{t\ge0}$ is bounded in $L^2_\sigma(\Omega)$, which implies the first line of \eqref{est.thm.L2L2}. From \eqref{est2.prop.resol.LpLq}, letting $t>0$ and $f\in L^2_\sigma(\Omega)$, 
\begin{align*}
\|\nabla e^{-t {\mathbb A}_V} f\|_{L^2}
&\le 
\varlimsup_{b\to0}
\int_{\gamma_b}
\|e^{t \lambda} \nabla(\lambda+{\mathbb A}_V)^{-1} f \|_{L^2} 
|\dd \lambda| \\
&\le 
C\|f\|_{L^2}
\int_{0}^{\infty} 
s^{-\frac12} e^{(\cos \phi)ts} 
\dd s, 
\end{align*}
which implies the second line of \eqref{est.thm.L2L2}. This completes the proof of Theorem \ref{thm.L2L2}.

\subsection*{Acknowledgements}
The author is partially supported by JSPS KAKENHI Grant Number JP 20K14345. 

\bibliography{Ref}

\begin{thebibliography}{10}

\bibitem{Abe(2020)}
Ken Abe.
\newblock Liouville theorems for the {S}tokes equations with applications to
  large time estimates.
\newblock {\em J. Funct. Anal.}, 278(2):108321, 30, 2020.

\bibitem{Abe(2021)}
Ken Abe.
\newblock On the large time {$L^\infty$}-estimates of the {S}tokes semigroup in
  two-dimensional exterior domains.
\newblock {\em J. Differential Equations}, 300:337--355, 2021.

\bibitem{Andrews-Askey-Roy(1999)}
George~E. Andrews, Richard Askey, and Ranjan Roy.
\newblock {\em Special functions}, volume~71 of {\em Encyclopedia of
  Mathematics and its Applications}.
\newblock Cambridge University Press, Cambridge, 1999.

\bibitem{Borchers-Miyakawa(1992)}
Wolfgang Borchers and Tetsuro Miyakawa.
\newblock {$L^2$}-decay for {N}avier-{S}tokes flows in unbounded domains, with
  application to exterior stationary flows.
\newblock {\em Arch. Rational Mech. Anal.}, 118(3):273--295, 1992.

\bibitem{Borchers-Miyakawa(1995)}
Wolfgang Borchers and Tetsuro Miyakawa.
\newblock On stability of exterior stationary {N}avier-{S}tokes flows.
\newblock {\em Acta Math.}, 174(2):311--382, 1995.

\bibitem{Borchers-Varnhorn(1993)}
Wolfgang Borchers and Werner Varnhorn.
\newblock On the boundedness of the {S}tokes semigroup in two-dimensional
  exterior domains.
\newblock {\em Math. Z.}, 213(2):275--299, 1993.

\bibitem{Brandolese-Schonbek(2018)}
Lorenzo Brandolese and Maria~E. Schonbek.
\newblock Large time behavior of the {N}avier-{S}tokes flow.
\newblock In {\em Handbook of mathematical analysis in mechanics of viscous
  fluids}, pages 579--645. Springer, Cham, 2018.

\bibitem{Chang-Finn(1961)}
I-Dee Chang and Robert Finn.
\newblock On the solutions of a class of equations occurring in continuum
  mechanics, with application to the {S}tokes paradox.
\newblock {\em Arch. Rational Mech. Anal.}, 7:388--401, 1961.

\bibitem{Dan-Shibata(1999a)}
Wakako Dan and Yoshihiro Shibata.
\newblock On the {$L_q$}--{$L_r$} estimates of the {S}tokes semigroup in a
  two-dimensional exterior domain.
\newblock {\em J. Math. Soc. Japan}, 51(1):181--207, 1999.

\bibitem{Dan-Shibata(1999b)}
Wakako Dan and Yoshihiro Shibata.
\newblock Remark on the {$L_q$}-{$L_\infty$} estimate of the {S}tokes semigroup
  in a {$2$}-dimensional exterior domain.
\newblock {\em Pacific J. Math.}, 189(2):223--239, 1999.

\bibitem{Drazin-Reid(2004)}
P.~G. Drazin and W.~H. Reid.
\newblock {\em Hydrodynamic stability}.
\newblock Cambridge Mathematical Library. Cambridge University Press,
  Cambridge, second edition, 2004.
\newblock With a foreword by John Miles.

\bibitem{Drazin-Riley(2006)}
P.~G. Drazin and N.~Riley.
\newblock {\em The {N}avier-{S}tokes equations: a classification of flows and
  exact solutions}, volume 334 of {\em London Mathematical Society Lecture Note
  Series}.
\newblock Cambridge University Press, Cambridge, 2006.

\bibitem{Farwig-Neustupa(2007)}
Reinhard Farwig and Ji\v{r}\'{\i} Neustupa.
\newblock On the spectrum of a {S}tokes-type operator arising from flow around
  a rotating body.
\newblock {\em Manuscripta Math.}, 122(4):419--437, 2007.

\bibitem{Galdi(2004)}
Giovanni~P. Galdi.
\newblock Stationary {N}avier-{S}tokes problem in a two-dimensional exterior
  domain.
\newblock In {\em Stationary partial differential equations. {V}ol. {I}},
  Handb. Differ. Equ., pages 71--155. North-Holland, Amsterdam, 2004.

\bibitem{Galdi(2011)}
Giovanni~P. Galdi.
\newblock {\em An introduction to the mathematical theory of the
  {N}avier-{S}tokes equations}.
\newblock Springer Monographs in Mathematics. Springer, New York, second
  edition, 2011.
\newblock Steady-state problems.

\bibitem{Galdi-Yamazaki(2015)}
Giovanni~P. Galdi and Masao Yamazaki.
\newblock Stability of stationary solutions of two-dimensional
  {N}avier-{S}tokes exterior problem.
\newblock In {\em The proceedings on {M}athematical {F}luid {D}ynamics and
  {N}onlinear {W}ave}, volume~37 of {\em GAKUTO Internat. Ser. Math. Sci.
  Appl.}, pages 135--162. Gakk\={o}tosho, Tokyo, 2015.

\bibitem{Guillod(2015)}
Julien Guillod.
\newblock Steady solutions of the {N}avier–{S}tokes equations in the plane.
\newblock {\em arXiv:1511.03938}, 2015.

\bibitem{Guillod(2017)}
Julien Guillod.
\newblock On the asymptotic stability of steady flows with nonzero flux in
  two-dimensional exterior domains.
\newblock {\em Comm. Math. Phys.}, 352(1):201--214, 2017.

\bibitem{Guillod-Wittwer(2015)}
Julien Guillod and Peter Wittwer.
\newblock Generalized scale-invariant solutions to the two-dimensional
  stationary {N}avier-{S}tokes equations.
\newblock {\em SIAM J. Math. Anal.}, 47(1):955--968, 2015.

\bibitem{Hamel(1917)}
Georg Hamel.
\newblock Spiralförmige bewegungen zäher flüssigkeiten.
\newblock {\em Jahresbericht der Deutschen Mathematiker-Vereinigung},
  25:34--60, 1917.

\bibitem{Heywood(1970)}
John~G. Heywood.
\newblock On stationary solutions of the {N}avier-{S}tokes equations as limits
  of nonstationary solutions.
\newblock {\em Arch. Rational Mech. Anal.}, 37:48--60, 1970.

\bibitem{Higaki(2019)}
Mitsuo Higaki.
\newblock Note on the stability of planar stationary flows in an exterior
  domain without symmetry.
\newblock {\em Adv. Differential Equations}, 24(11-12):647--712, 2019.

\bibitem{Higaki(2022)}
Mitsuo Higaki.
\newblock Existence of planar non-symmetric stationary flows with large flux in
  an exterior disk.
\newblock {\em arXiv:2207.11922}, 2022.

\bibitem{Higaki(2023)}
Mitsuo Higaki.
\newblock Stability of a two-dimensional stationary rotating flow in an
  exterior cylinder.
\newblock {\em Math. Nachr.}, 296(1):314--338, 2023.

\bibitem{Hillairet-Wittwer(2013)}
Matthieu Hillairet and Peter Wittwer.
\newblock On the existence of solutions to the planar exterior {N}avier
  {S}tokes system.
\newblock {\em J. Differential Equations}, 255(10):2996--3019, 2013.

\bibitem{Hishida(2016)}
Toshiaki Hishida.
\newblock Asymptotic structure of steady {S}tokes flow around a rotating
  obstacle in two dimensions.
\newblock In {\em Mathematical fluid dynamics, present and future}, volume 183
  of {\em Springer Proc. Math. Stat.}, pages 95--137. Springer, Tokyo, 2016.

\bibitem{Karch-Pilarczyk(2011)}
Grzegorz Karch and Dominika Pilarczyk.
\newblock Asymptotic stability of {L}andau solutions to {N}avier-{S}tokes
  system.
\newblock {\em Arch. Ration. Mech. Anal.}, 202(1):115--131, 2011.

\bibitem{Kato(1976)}
Tosio Kato.
\newblock {\em Perturbation theory for linear operators}.
\newblock Grundlehren der Mathematischen Wissenschaften, Band 132.
  Springer-Verlag, Berlin-New York, second edition, 1976.

\bibitem{Kozono-Ogawa(1993)}
Hideo Kozono and Takayoshi Ogawa.
\newblock Two-dimensional {N}avier-{S}tokes flow in unbounded domains.
\newblock {\em Math. Ann.}, 297(1):1--31, 1993.

\bibitem{Kozono-Sohr(1992)}
Hideo Kozono and Hermann Sohr.
\newblock On a new class of generalized solutions for the {S}tokes equations in
  exterior domains.
\newblock {\em Ann. Scuola Norm. Sup. Pisa Cl. Sci. (4)}, 19(2):155--181, 1992.

\bibitem{Ladyzhenskaya(1969)}
O.~A. Ladyzhenskaya.
\newblock {\em The mathematical theory of viscous incompressible flow}.
\newblock Mathematics and its Applications, Vol. 2. Gordon and Breach Science
  Publishers, New York-London-Paris, 1969.
\newblock Second English edition, revised and enlarged, Translated from the
  Russian by Richard A. Silverman and John Chu.

\bibitem{Lunardi(1995)}
Alessandra Lunardi.
\newblock {\em Analytic semigroups and optimal regularity in parabolic
  problems}, volume~16 of {\em Progress in Nonlinear Differential Equations and
  their Applications}.
\newblock Birkh\"{a}user Verlag, Basel, 1995.

\bibitem{Maekawa(2017a)}
Yasunori Maekawa.
\newblock On stability of steady circular flows in a two-dimensional exterior
  disk.
\newblock {\em Arch. Ration. Mech. Anal.}, 225(1):287--374, 2017.

\bibitem{Maekawa(2017b)}
Yasunori Maekawa.
\newblock Remark on stability of scale-critical stationary flows in a
  two-dimensional exterior disk.
\newblock In {\em Mathematics for nonlinear phenomena---analysis and
  computation}, volume 215 of {\em Springer Proc. Math. Stat.}, pages 105--130.
  Springer, Cham, 2017.

\bibitem{Maekawa-Tsurumi(2023)}
Yasunori Maekawa and Hiroyuki Tsurumi.
\newblock Existence of the stationary navier-stokes flow in $\mathbb{R}^2$
  around a radial flow.
\newblock {\em Journal of Differential Equations}, 350:202--227, 2023.

\bibitem{Maremonti-Solonnikov(1997)}
P.~Maremonti and V.~A. Solonnikov.
\newblock On nonstationary {S}tokes problem in exterior domains.
\newblock {\em Ann. Scuola Norm. Sup. Pisa Cl. Sci. (4)}, 24(3):395--449, 1997.

\bibitem{Masuda(1984)}
Ky\={u}ya Masuda.
\newblock Weak solutions of {N}avier-{S}tokes equations.
\newblock {\em Tohoku Math. J. (2)}, 36(4):623--646, 1984.

\bibitem{Russo(2011)}
Antonio Russo.
\newblock On the existence of {D}-solutions of the steady-state navier-stokes
  equations in plane exterior domains.
\newblock {\em arXiv:1101.1243}, 2011.

\bibitem{Sohr(2013)}
Hermann Sohr.
\newblock {\em The {N}avier-{S}tokes equations: An elementary functional
  analytic approach}.
\newblock Modern Birkh\"{a}user Classics. Birkh\"{a}user/Springer Basel AG,
  Basel, 2001.

\bibitem{Watson(1944)}
G.~N. Watson.
\newblock {\em A {T}reatise on the {T}heory of {B}essel {F}unctions}.
\newblock Cambridge University Press, Cambridge, England; The Macmillan
  Company, New York, 1944.

\bibitem{Yamazaki(2016)}
Masao Yamazaki.
\newblock Rate of convergence to the stationary solution of the
  {N}avier-{S}tokes exterior problem.
\newblock In {\em Recent developments of mathematical fluid mechanics}, Adv.
  Math. Fluid Mech., pages 459--482. Birkh\"{a}user/Springer, Basel, 2016.

\end{thebibliography}
\bibliographystyle{plain}

\medskip

\begin{flushleft}
M. Higaki\\
Department of Mathematics, 
Graduate School of Science, 
Kobe University, 
1-1 Rokkodai, 
Nada-ku, 
Kobe 657-8501, 
Japan.
Email: higaki@math.kobe-u.ac.jp
\end{flushleft}

\medskip

\noindent \today

\end{document}